\documentclass[12pt]{amsart}

\usepackage[english]{babel}
\usepackage[utf8]{inputenc}
\usepackage[T1]{fontenc}

\usepackage[margin=1.5in]{geometry}

\usepackage[colorinlistoftodos]{todonotes}
\usepackage[colorlinks=true, allcolors=blue]{hyperref}

\usepackage[alphabetic]{amsrefs}

\usepackage[makeroom]{cancel}
\usepackage{setspace}
\usepackage{amsfonts} 
\usepackage{amssymb}
\usepackage{amsthm}
\usepackage{indentfirst}
\usepackage{enumerate}
 \usepackage{enumitem}   
\usepackage{amsmath}
\usepackage{graphicx}

\usepackage{mathabx}
\graphicspath{ {./images/} }

\usepackage{enumitem}
\usepackage[most]{tcolorbox}

\usepackage{tikz-cd}

\definecolor{myGreen}{RGB}{10 100 10}

\newcommand{\bC}{\mathbb{C}}

\newcommand{\bR}{\mathbb{R}}

\newcommand{\bF}{\mathbb{F}}
\newcommand{\bN}{\mathbb{N}}

\newcommand{\bA}{\mathbb{A}}

\newcommand{\cD}{\mathcal{D}}

\newcommand{\cS}{\mathcal{S}}

\newcommand*{\tran}{\mathsf{T}}

\newtheorem{theorem}{Theorem}[section]
\newtheorem{lemma}[theorem]{Lemma}
\newtheorem{prop}[theorem]{Proposition}

\newtheorem{cor}[theorem]{Corollary}
\newtheorem{thm}[theorem]{Theorem}
\newtheorem{lem}[theorem]{Lemma}

\newtheorem*{cor*}{Corollary}
\newtheorem*{thm*}{Theorem}
\newtheorem*{lem*}{Lemma}

\newtheorem*{prop*}{Proposition}

\theoremstyle{definition}
\newtheorem{definition}[theorem]{Definition}

\newtheorem*{defn*}{Definition}

\theoremstyle{remark}

\title[Base norm spaces--classical and nc]{\textbf{Base norm spaces--classical, complex, and noncommutative}}

\author{David P. Blecher}
	
\address{Department of Mathematics, University of Houston, Houston, TX 77204-3008.}
	
\email{dpbleche@central.uh.edu}

\author{Damon M. Hay}
	
\address{Department of Mathematics \& Statistics, Sam Houston State University, Huntsville, TX 77341}
	
\email{dhay@shsu.edu} 

\date{Revision of 2/12/2026}

\begin{document}
    \begin{abstract}    We generalize the theory of base norm spaces to the complex case, and  further to the noncommutative setting relevant to `quantum convexity'.  In particular, 
    we establish the duality between complex Archimedean order unit spaces and complex base norm spaces, as well as the corresponding duality between their noncommutative counterparts. Additional topics include an exploration of natural connections with various notions of quantum convexity and regularity of noncommutative convex sets, and an analysis of how these concepts interact with complexification.    We also define, as in the classical case, a class that contains and generates the noncommutative base norm spaces, but is defined by fewer axioms.    We show how this may be applied to provide new and  interesting  examples of noncommutative  base norm spaces. 
  \end{abstract} \maketitle

\section{Introduction} 
In classical functional analysis, {\em base norm spaces} appear as the objects that are dual to {\em archimedean order unit spaces}.  Both are ordered vector spaces whose order structure induces a norm. Whereas archimedean order unit spaces (or aou spaces for short) carry a norm which is induced by an order unit, the norm on a base norm space is induced by a base for the cone of positive elements. The complete theory of these spaces also incorporates (Banach) dual versions of these objects. Aspects of this theory go back to at least the 1930s and 1940s with work of M. Krein, D. A. Edwards, J. Grosberg (see e.g. \cites{Krein,E,GK}), and certainly others, and continued into the 1960s when A.J. Ellis  essentially completed this theory in its classical (real) form (in e.g.\ \cite{Ellis}). 
In this note, we extend this theory to both complex ordered spaces and to their noncommutative counterparts, introducing the notion of a {\em complex base norm space} and a {\em noncommutative (nc) base norm space} (both real and complex), and by proving the key duality results about these spaces. Complex and noncommutative analogues of archimedean order units are already objects of study.  Indeed, it is well-known that the concept of an {\em operator system} is a far-reaching noncommutative generalization of an aou space which plays a central role in the theory of operator algebras and noncommutative functional analysis, generally.  

In classical functional analysis, the dual of a real order unit space is a real base-norm space, and vice versa.  We prove the analogue of this in the complex case and in the noncommutative setting.  In particular,  the operator space dual of an operator system is a nc base norm space, and vice versa, including variants for dual versions of these spaces.
In some sense this means that 
the various base norm spaces are simply the preduals or duals of operator systems, and `bases' are just state spaces.  Nonetheless they are interesting operator spaces.
Moreover the concept of a base, and the base norm theory,  is valuable, for example in constructing new examples, or for  connecting to convexity theory where bases are a central object.  Or for  connecting to parts of the quantum physics literature such as GPT's (Generalized Probability Theories).
This was one of the original motivations for studying abstract base norm and aou spaces in the 70's.  E.g. this was initiated perhaps in \cite{DL}, and continues until today (see e.g.\ \cite{Lami,ALetal}, which also describe some of the older work).  In the words of the latter reference, GPT's encompass  ``all physical models whose predictive power obeys minimal requirements'' and ``a GPT makes the host vector space a Banach space in a canonical way, by equipping it with a so-called base norm".   Another example:\  from the perspective of most physicists working in quantum information theory (QIT), {\em Quantum channels} are just our {\em base morphisms}, that is maps between matrices (viewed as nc base spaces in our language) preserving the base.  Indeed such physicists usually prefer the base formulation, as any operator algebraist not familiar with physics knows who has tried to read a quantum physics or QIT article and found themselves having to `translate'  the predual or base formulations into statements at the algebra level.
We have not taken the time to do so here but parts of our paper can be related to entanglement, analogously to e.g.\ \cite{PTT} and other works building on that paper (e.g.\ \cite[Section 9]{BR} and references therein).  

One may view classical or noncommutative base norm spaces 
as a subtopic of classical or noncommutative convexity. 
By Kadison's representation theorem (see below), complete aou spaces `are' just spaces $A(K)$ of continuous scalar-valued affine functions  on a compact convex set $K$. Here the compact convex set is the state space of the order unit space. This persists in the noncommutative setting as well, where one considers matrix affine functions $\bA(K)$ and matrix or nc convex sets
\cite{WW,DK,BMcI}.  Regardless, this convex set becomes the base of the cone for the dual space. This relation is reversed when one starts with a base norm space and forms its dual. 
These hold in either the real or the complex setting.  

We now describe the structure of our paper and some of the main results not already alluded to. We first explain notation and give background information at the conclusion of this section.  
In Section \ref{rbns} we review and clarify terminology for real base norm spaces, and address some  
variations in the literature regarding the definition of a base norm space, especially on whether or not the positive cone or base are assumed closed.  We describe this a bit more below.  In Section \ref{cbns}, the notion of a complex base norm space is introduced.  In this section we are doing for classical base norm spaces what was done for aou spaces in parts of \cite{PT}.   This is based on the complexification of a real base norm space using the \emph{dual Taylor norm} which we describe there.  
With this in mind the theory of complex base norm spaces becomes a simple add-on to the classical theory of real base norm spaces from Section \ref{rbns} or \cite{AE,Alfsen}. Essentially everything in the complex theory follows quickly from the 
real theory, together with the fact above that the complex base norm is completely determined, via the dual Taylor norm,  by the classical `base norm' on the selfadjoint part.  Or in other words, the complex case is just a standard complexification of the real case, as is also the case for aou spaces.   (Nonetheless there are some slight differences between these two theories.) 
We prove the fundamental duality results indicated above in this setting.

In Section \ref{ncc} we introduce our variant of a noncommutative base norm space, and prove the natural duality theorems along with other relevant results and corollaries.  
We also define, as in the classical case, a class called {\em matrix base ordered spaces} that contains and generates the noncommutative base norm spaces, but is defined by fewer axioms.  
One of the conditions defining this class was inspired by a recent approach of Travis Russell to `noncommutative base norm spaces' \cite[Section 7]{Rus} which we discovered only after the first ArXiV version of our paper. (In the earlier version we had a longer definition of matrix base ordered spaces, but we were able to simplify it by Russells's idea and our new Lemma \ref{ts}.)   Russell's spaces are very interesting, but are not a strict generalization of classical 
base norm spaces as we discuss in the Acknowledgements.   As in the classical case, we show that we can `complete' any matrix base ordered space to become a nc base norm space.  This construction will be important in 
constructing new examples of  nc base norm spaces in the future.  
We also explore, for example,  how these concepts interact with complexification; and with various notions of `quantum convexity',  and with `regularity of nc convex sets' developed in a companion paper \cite{Breg}. 

Section 
\ref{ex} is devoted to some examples.
Subsection \ref{ex1}  takes a detailed look at the Paulsen system of an operator space as an example of a nc base norm space.  This is interesting because it is both an infinite dimensional operator system and a nc base norm space with the same matrix order, but with different, though equivalent, matrix norms.   Also, finite dimensional operator systems are  nc base norm spaces.  This adds a perspective to the known Choi-Effros-Paulsen operator system duality theory, whose real version is summarized in \cite[Section 8]{BR}. (In the latter connection we make a remark concerning an objection that might be raised by some readers familiar with the 
latter theory.  Namely, some might say that ``the modern approach is that the 
dual of an operator system should be an operator system, not a
base norm space''.  
This objection is in some sense an example of the ``false dilemma" fallacy.
Indeed for many (indeed very noble) purposes one really does want (some) duals of  operator systems to be operator systems, but there are other areas of our field where one does not want this, in particular where we definitely do not want 
to lose the dual space norm on the dual space.  To illustrate this forcibly, consider the dual of $\ell^\infty_n$.  Its 
canonical Choi-Effros operator system dual is $\ell^\infty_n$ again.  However in many problems in analysis we
(and statistically speaking, most analysts or physicists) really care about the actual $\ell^1_n$ norm, and would not want to lose it!  A similar situation exists for $M_n$. 
Besides, in this section we are in part showing that these two views are compatible.)

In this section we also give a new approach  to, and definition of, complex base norm spaces, via the nc theory. 

It is worth saying up front that there are several variants of the theory developed here. The reader may find this bewildering at first, but in fact all of these variants are beautifully interrelated (by theorems here).   Indeed, as we will explain below, even in the classical real theory there are two main distinct classes both called `base norm space',
one of which has no norm mentioned in its definition.
Also, as is typical in functional analysis there are distinct  norm and weak* versions; thus we will consider 
dual base norm spaces and their noncommutative variants.
We use the plural `variants' here because we will consider
both the (Wittstock) matrix convexity and the (Davidson-Kennedy) nc convexity variants.  Finally, there is the real and complex distinction; these two are interrelated by complexification.   

In the literature on classical (real) base norm spaces there is a curious phenomena.   By far the main definition usually given 
of base norm spaces (e.g.\ in \cites{AE,AS1}) is located within  the category of ordered normed vector spaces, so that the base and positive cone is assumed closed, etc.   However it is often pointed out that there is
a {\em more general class} (which is usually given the same name), defined purely within in the category of ordered vector spaces,  so there is no norm or topology in the definition.  This class is important partly because it is often easier to construct examples of these, because the definition is so simple.  Then there is a third topological variant considered in \cite{Jam}, and a fourth more restrictive class considered in \cite{Alfsen}.   We will not consider the last two classes because they are not usually relevant to us. E.g.\ many interesting real base norm spaces in our first sense (the sense of  \cite{AE}) are not base norm spaces in the fourth sense.  We call a space $X$ in 
the second class a 
{\em pre-base norm space}.  Then $X$ is known to have a canonical `base norm', however the base and positive cone can fail to be closed 
in this norm.  Thus this class of ordered spaces is strictly larger.   We give the class a name, unlike in the case of
aou spaces, since it is well known that for the latter the cone is automatically closed \cite{AS1}. 
However the `completion' of $X$ 
is a base norm space in the first sense (see  Theorem \ref{xpc} and the lines thereafter).  An exactly similar phenomenon occurs in the noncommutative case, 
as explained above in our description of Section \ref{ncc} there is
a {\em more general} class, the {\em matrix base ordered spaces} which is defined using fewer  axioms.

Turning to background and notation, we will be considering real and complex vector spaces, which may also be normed spaces, ordered vector spaces, or $*$-vector spaces.
We write $\bF$ for either $\bR$ or $\bC$, and $H$ for a Hilbert space. An {\em ordered vector space} is a vector space with a proper positive cone. For an ordered vector space $E$, the cone of positive elements will be denoted by $E_+$. By a {\em $*$-vector space} we mean a vector space with an involution 
(a period 2 automorphism).  If $\bF = \bC$ we assume that the involution is conjugate linear. If $A$ is a set in a $*$-vector space we denote the set of selfadjoint elements in $A$ by $A_{\rm sa}$. 
 We assume that the positive cone of an ordered $*$-vector space $E$ is contained in the selfadjoint part of the space. 
 We recall that the positive cone is called {\em generating} if $E_{\rm sa} = E_+ - E_+$.
 This is equivalent to every element of $E_{\rm sa}$ being dominated by an element of $E_+$. 
 An {\em order unit} for an ordered vector space is an element $e$ such that for all selfadjoint $x$, there exist real $r > 0$ such that $re \ge x$. The order unit is {\em archimedean} if $re + x \ge 0$ for all $r >0$ implies $x$ is positive. A real {\em archimedean order unit space} is a real ordered vector space with an archimedean order unit.   It admits a norm defined by 
$\| a \| = \inf \{ t > 0 : -t e \leq a \leq t e \}$, the {\em order unit norm}.    
 We will discuss {\em complex archimedean order spaces} later.   An example of a real (resp.\ complex) aou space is a unital {\em function space/system}, that is a unital subpace (resp.\ selfadjoint subspace) 
   of $C(\Omega)$ for compact Hausdorff 
 $\Omega$. 
 
We write $M_n(E)$ for the $n \times n$ matrices with entries from a vector space $E$.  If $n$ is infinite and $E$ is an operator space we will mean by $M_n(E)$ the matrices whose   finite submatrices have uniformly bounded norm.  We will write $M_n$ for infinite $n$ for $M_n(\bF)$ in this sense. Given $x \in M_n(E)$, we will sometimes use the notation $\tilde{x}$ to represent the $2n \times 2n$ selfadjoint block matrix with $x$ in the $1$-$2$-block and $x^*$ in the $2$-$1$-block.  

We assume that the reader is familiar with basic convexity theory.  We denote the convex hull of a set $A$ by ${\rm co} (A)$ and the closed convex hull by $\overline{{\rm co}} (A)$.  The {\em absolute convex hull} of a set $A$ is the collection of all `absolute convex combinations' $\sum_{i=1}^n t_ix_i$, where $x_i \in A$ and $t_i \in \bC$ with $|t_1| + \cdots +|t_n| \le 1$.

An \emph{operator space} is a subspace of $B(H)$, the bounded operators on a Hilbert space $H$, or abstractly it is a 
a vector space $E$ with a norm $\| \cdot \|_n$ on $M_n(E)$ for each $n \in \bN$ satisfying the axioms of Ruan's characterization (see e.g. \cite{ERBook}).  If $X \subset B(H)$ is an operator space, then the `matrix norms' above are given by identifying $M_n(B(H))$ with the bounded operators on $n$-fold direct sum of copies of $H$. If $T : X \to Y$ we write $T^{(n)}$ for the canonical `entrywise' amplification taking $M_n(X)$ to $M_n(Y)$, i.e. $T^{(n)}([x_{ij}]) = [T(x_{ij})]$.  
The completely bounded norm is $\| T \|_{\rm cb} = \sup_n \, \| T^{(n)} \|$, and $T$ is completely  contractive if  $\| T \|_{\rm cb}  \leq 1$. A \emph{hyperplane} in a vector space $E$ will be a set of the form $\{x \in E: f(x) = 1\}$ for a linear functional $f$ on $E$. A {\em nc hyperplane} in $E$ is the sequence $(H_n)$ of sets $H_n = \{x \in M_n(E): f^{(n)}(x) = I_n\}$, where $f$ is a fixed linear functional on $E$, and $I_n$ is the $n \times n$ identity matrix.
 
An {\em operator system} is a unital selfadjoint subspace of $B(H)$.  
We denote the identity operator in $M_n(B(H))$ by $I$ or $I_n$.
A map $T$ is said to be {\em positive} if it takes positive elements to positive elements, and  {\em completely positive} if $T^{(n)}$ is  positive for all $n \in \bN$. A {\em ucp map} is  unital, linear, and completely positive. Of course $T$ is {\em selfadjoint} if $T(x^*) = T(x)^*$ for $x \in X$.  This is automatic for completely positive maps between real or complex operator systems.   A \emph{state} on an operator system or aou space $V$ is a (selfadjoint) positive unital (scalar valued) functional,  or equivalently a contractive unital functional, and $S(V)$ is the (compact convex) set of states.  
Any  function system is an (abstract) operator system. 

For general background on operator systems and spaces, we refer the reader to e.g. \cite{Pnbook,P,BLM,DK} and in the real case to e.g.\ \cite{BR,BMcI}.
The theory of complex   $C^*$- and von Neumann algebra theory may be found in e.g.\ \cite{Ped}, and 
 basic real $C^*$- and von Neumann algebra theory in \cite{Li}.  
We write $A(K)$ or $A(K,\bF)$ for the continuous affine scalar functions on a compact convex set $K$, which are unital selfadjoint subspaces of $C(K, \bF)$, the continuous functions on $K$ with values in the field $\bF = \bR$ or $\bC$. We write $\bA(K)$ or $\bA_{\bF}(K)$ for the noncommutative version from \cite{DK,BMcI}. We  write  ${\rm UCP}^\sigma(V, M_n)$ for the collection of weak* continuous ucp maps into $M_n$, that is, the normal matrix state space of $V$.

Complex operator systems are well studied and understood \cite{Pnbook}.  Their connection to matrix convex (resp.\ nc convex) sets in a complex operator space  may be found in \cite{WW,DK,Dav}. The theory of real operator systems may be found in \cite{BR}.  
A (real or complex) {\em matrix convex set} in a (real or complex)  
vector space $E$ is a sequence $K = (K_n)$ of subsets of $M_n(E)$ satisfying 1)\ $x \in K_m$ and $y \in K_n$ implies $x \oplus y \in K_{m+n}$, and
  2)\ $a \in M_{n,m}(\bF)$ with $a^* \, a = I_n$ and $x \in K_n$ implies $a^*xa \in K_m$.   Here $n, m \in \bN$.  We call $K_n$ the $n$th level. 
  Real matrix convex (resp.\ nc convex) sets are studied in e.g.\ \cite{EPS} (resp.\ \cite{BMcI}).   
   If $E$ is a topological vector space, then we say $K$ is closed (compact) if each $K_n$ is closed (compact).
  The matrix state space   $({\rm UCP}(V, M_n))$ of an operator system is the generic example of a compact matrix convex set \cite{WW}. 
A (real or complex) {\em matrix ordered space} is a $*$-vector space $X$ with a proper cone $M_n(X)_+ \subset M_n(X)_{\rm sa}$ such that $\alpha^* M_n(X)_+  \alpha 
\subset M_m(X)_+$ for $\alpha \in M_{n,m}(\bF)$. 
It follows that 
$M_n(X)_+  \oplus M_m(X)_+ \subset M_{n+m}(X)_+$. 
We will say that such $E$ is 
a {\em matrix ordered matrix normed space}
if in addition there is a norm defined at each level,  the matrix cones $M_n(E)_+$ are closed, and $\| x^* \|_n = \| x \|_n$ for $x \in M_n(X)$. An {\em archimedean matrix order unit space} is a matrix ordered space with an archimedean order unit $e$ such that $e \otimes I_n$ is an archimedean order unit for each matrix level.  We will use the abstract characterization  of operator systems as 
matrix ordered matrix normed space with an archimedean matrix order unit
(due to Choi and Effros in the complex case \cite{CE}, see e.g.\ \cite[Section 2]{BR} for discussion of the real case). 

As we near the end of this segment on background information, we highlight Kadison's representation theorem as a crucial piece of the theory. It was originally stated for real aou spaces (see \cite{Alfsen,AS1} and e.g.\ Section 4.3 of  \cite{KRI}), but  it extends immediately to the complex case  (see \cite{PT}, or \cite[Lemma 1.2]{Breg} and the lines after that for a short selfcontained proof). This theorem and its noncommutative analogues provide the dual equivalencies between the category of aou spaces and the category of compact convex sets, and between the category of operator systems and the category of compact matrix convex sets, respectively.

\begin{theorem} {\rm (Kadison's theorem)}  Archimedean order unit real vector spaces (resp.\ complex  $*$-vector spaces) $V$ are exactly (i.e.\ are unitally  order isomorphic to) 
the real (resp.\  complex) function systems.   The subclass of these whose selfadjoint part is complete in the canonical norm coincides  up to unital  order isomorphism with 
 the $A(K)$ spaces, for a compact convex set $K$.
Thus if  $V$ is a  complex $*$-vector space such that $(V_{\rm sa}, V_{+},e)$ is a real Archimedean order unit space, then $V$ is (complex $*$-linearly  unital order  embedded as) a selfadjoint unital subspace of 
         $C(K,\bC)$ for a compact Hausdorff set $K$.   
         Indeed $K$ may be chosen to be 
         convex, and the latter subspace to be $A(K,\bC)$
         if $V_{\rm sa}$ is complete in the order norm. Moreover,   the embedding/isomorphism  may be chosen such that it is a linear isometry with respect to  the norm on $V$ induced by its state space. 
\end{theorem}

The just mentioned norm on a complex aou space $V = V_{\rm sa} + i V_{\rm sa}$ is called the 
 {\em minimal order unit norm} or {\em extended order unit norm}.   We will sometimes call a complex aou space with this extended  norm
   a {\em normed complex aou space}.
It will be useful to us that another way to view this norm  is as exactly the 
{\em Taylor norm} of the complexification of $V_{\rm sa}$ 
\cite{MMPS}.   That is,  it is exactly the well known Banach space injective tensor  norm: 

 \begin{lem} \label{ist}  The canonical (minimal order unit) norm on  a complex  aou space $V$  is the injective tensor product norm   from the  isomorphism $V \cong V_{\rm sa} \, \overset{\vee}{\otimes}_{\bR}  \, l^2_2(\bR)$. Moreover, any norm on a complex aou space $V$ inherited from a selfadjoint unital order embedding $\rho$ of $V$ in a $C(K,\bC)$ space, which is
the order unit norm on $V_{\rm sa}$, is this same norm. 
\end{lem} 

\begin{proof}  We simply sketch this, since there are several ways to see this which the reader might prefer, 
for example using the simple fact \cite{MMPS} that the 
Taylor norm is the `unit circle norm' $\| x + i y \|_T = \sup \{ \| sx + t y \| : 
s^2 + t^2 = 1 \}$. 
Alternatively, suppose that $V_{\rm sa}$ is a unital subspace of $C(K,\bR)$.  We can take $K$ to be the real state space of $V_{\rm sa}$.  Then the canonical embedding of $V$ in $C(K,\bC) = C(K,\bR) \, \overset{\vee}{\otimes}_{\bR}  \, \bC$  is isometric for the canonical norm on  $V$ (e.g.\ see \cite[Lemma 1.2]{Breg}).  This gives the first assertion because the injective   norm is `injective'. 
For the last part, the restriction 
of 
$\rho$ to $V_{\rm sa}$ (and more generally any real unital order embedding of  $V_{\rm sa}$ 
in a $C(K,\bR)$), is an isometry \cite[Corollary II.1.4]{Alfsen}.  Hence so is $\rho$ (e.g.\ by complexifying).  We see that the norm of $\rho(x)+i\rho(y)$ for $x, y \in V_{\rm sa}$ is the above norm.  
\end{proof} 

The next  lemma establishes an identification that will be used repeatedly.

 \begin{lem} \label{sad}  If $V$ is a complex  operator system (or  
 complex  normed $*$-vector space)
 then $(V^*)_{\rm sa} \cong (V_{\rm sa})^*$ 
 real 
 isometrically (and weak* homeomorphically) via the map
 $\psi \mapsto {\rm Re} \, \psi$.  The first `$*$' here is the complex  dual, while the second is
 the real dual.  If $V_{\rm sa}$ is an ordered vector space then 
 the latter map is an order isomorphism.  
\end{lem} 

\begin{proof}  The expression $\frac{1}{2} (x + x^*)$ defines a real
quotient map 
$V \to  V_{\rm sa}$ whose kernel is the skew elements of $V$.   Dualizing we obtain a real  weak* 
continuous complete isometry $(V_{\rm sa})^* \to (V_r)^*$ with range $W$, the annihilator of the  skew elements. Here, the subscript '$r$' indicates that the space is being regarded as a real vector space.
Thus, $W$ consists of the real continuous functionals with $\psi(x^*) = \psi(x)$ for all $x \in V$.
However, $W$ is taken onto $(V^*)_{\rm sa}$ by
the canonical isometry $(V_r)^* \to (V^*)_r$ (see \cite[Proposition 1.1.6]{Li}).  For example, the real part of any $\psi \in (V^*)_{\rm sa}$ is clearly in $W$.
If $V_{\rm sa}$ is an ordered vector space and $\psi \in (V^*)_{\rm sa}$
then $\psi \geq 0$ on the positive cone if and only if ${\rm Re} \, \psi$ is 
(since they agree there). 
\end{proof} 

Finally, we conclude this section with remark 
about the link between the positive cones in 
a normed (or matrix normed) ordered space
$X$ and in its dual, which we will use silently throughout the paper.  In particular it will confirm  that the dualities between operator systems and nc base norm  spaces proved below are perfect.   
Of course it is well known that positive cones $P$ for $X$ induce positive cones $P^*$ for $X^*$, and vice versa.
We already used some of this in the last lemma. 
If $P$ is a cone in a real normed space $X$ (resp.\ in $X^*$) then
the induced cone in $X^*$ (resp.\ $X$) is the negative polar
(resp.\ prepolar), and by the bipolar theorem 
$$(P^{\circ})_{\circ} = \bar{P}.$$  
E.g.\ see 
3.1.7 in \cite{Jam} in the real case. So there is a perfect duality between these cones.  In the complex case a similar statement holds if $X$ is also a $*$-vector space with $X_+ \subset X_{\rm sa}$.  Indeed this follows from the real case since we can just work in the selfadjoint part, by courtesy of Lemma \ref{sad}.  Similar statements hold in the (complex or real) matrix ordered case (and for simplicity if the matrix cones $P$ are closed).  That is $(P^*)_* = P,$ where the set on the left is the cone induced on $X$ by $P^*$. This uses e.g.\ the separation/bipolar
theorem of Effros and Winkler \cite{EW}.    In \cite{BMcI} we checked that the proof of the latter also works in the real case.   Thus we have in both cases:

\begin{prop} If $P$ are the matrix cones for a matrix ordered matrix normed real or complex $*$-vector space then $(P^*)_* = P$ (at every finite matrix level). \end{prop}

\begin{proof} The proof 
follows familiar lines.  Clearly $P = (M_n(X)_+) \subset (P^*)_*$.
Suppose that $v_0$ lives in the $n$th level of $(P^*)_*$
but $v_0 \notin M_n(X)_+$. By the Effros and Winkler result above there exists  $\psi : X \to M_n$ satisfying Re $\psi^{(k)} \leq I_{nk}$ on $M_k(X)_+$ for all $k \in \bN$, but Re $\psi^{(n)}(v_0) \nleq 0$.  Let $\varphi = \frac{1}{2}(\psi + \psi^*)$.  Then $\varphi^{(k)} \leq I_{nk}$ on $M_k(X)_+$ for all $k \in \bN$.  In particular, 
for each $y \in M_k(X)_+$ and $\xi \in \bF^{nk}$
 we have 
 $$\langle \varphi^{(k)}(y) \xi , \xi \rangle 
 \leq \| \xi \|^2.$$
Since $M_k(X)_+$ is a cone it follows that 
$\langle \varphi^{(k)}(M_k(X)_+) \xi , \xi \rangle \leq 0$.
Thus $-\varphi$ is completely positive, and we obtain $$(-\varphi)^{(n)}(v_0)
= -{\rm Re} (\psi^{(n)}(v_0)) \geq 0,$$
a contradiction.  \end{proof} 

A similar result $(P_*)^* = P$ holds even at infinite matrix levels if $X = F^*$ is also a dual operator space, and the cones are weak* closed.   The proof is very similar,
but using the nc separation theorem of Theorem 2.4.1 of \cite{DK} (Corollary 16.5.3 of \cite{Dav}), or \cite[Theorem 3.6]{BMcI} in the real case.

\section{Real base norm spaces} \label{rbns} 
 
Suppose that we have an ordered real vector space  $X$  with  positive cone  $X_+ \subset X$. A convex set $K$ in 
 $X_+$ is called a {\em base} for $X_+$ if  for every nonzero $x \in  X_+$ there 
 is a unique scalar $t \geq 0$ and unique $k \in K$  with $x = tk$.  We  will also assume (1) that $X = X_+ - X_+$, or equivalently that  $X$ is real spanned by $K$, so that any $x \in X$
 may be written as $c_1 k_1 - c_2 k_2$ for $c_i \geq 0$ and $k_i \in K$.   This allows us to define an additive positive 
 scalar homogeneous function on $X_+$ by $f_1(x) = t$ if $x = tk$ as above. 
 Then it  is easily checked that $f_1$  extends uniquely to a well defined
 strictly positive (these are sometimes called faithful) real linear functional on $X$ with $f_1( c_1 k_1 - c_2 k_2) = c_1  - c_2$ (\cite[Lemma 2]{Ellis} or \cite[Section 3.9]{Jam}).
 We call 
$f_1$ the {\em base function}. 
 Conversely, if $f$ is a strictly positive real linear functional on $X$ then $f^{-1}( \{ 1 \})$ is a base.
 Let $C = {\rm co}(K \cup (-K))$.  We will also assume (2)\ that $C$ is {\em linearly bounded}: that is for all nonzero $x \in C$ there exists $n \in \bN$ with $nx \notin C$. If all the above hold (that is, if $X$ is an ordered real vector space  $X$  with  positive cone  $X_+ \subset X$ having a base $K$ satisfying (1) and (2))   then 
 we call $X$ a {\em real pre-base norm space}.  This space has a canonical norm on it described as follows. 
 It is well known that the Minkowski functional $p_C$ of  $C$ is a norm $\| \cdot \|$ on $X$ (see \cite[Lemma 3]{Ellis} or \cite[Section 3.9]{Jam}), and that $C$ is a norm dense subset of the 
 closed unit ball in this norm (indeed $C$ contains the open unit ball).  We call
 this norm the {\em base norm}.  Note that 
  $\| x \| = 1$ for $x \in K$ (see the just cited sources). 
 Thus $\| x \|  = f_1(x)$ for $x \in X_+$, and so 
 $f_1$ is continuous, indeed contractive: $|f_1(x - y)| = |\| x \| - \| y \| | \leq \| x - y \|$ for $x, y \in X_+$.
It is clear that  Ball$(X)  \subseteq t \, 
{\rm co} (K \cup (-K) )$ for all $t > 1$.  
Indeed $\| u \| = \inf \, (c_1 + c_2)$, the infimum taken over positive constants $c_i$ with $u = c_1 k_1 - c_2 k_2$ for some $k_i \in K$.  
 If $X_+$ is closed in the base norm then so is $K = X_+ \cap f_1^{-1}( \{ 1 \})$. 

 Although we have defined the base norm, at present our (pre-)base norm spaces are not ordered normed spaces yet in the usual sense.  That is
 $X_+$ may not be closed with respect to the canonical base norm (we have a simple counterexample).
 To fix this, we will show  shortly that the closure of $X_+$  is a proper cone which has the norm closure $\bar{K}$ of $K$ as a base making $X$ a pre-base norm space with closed cone, 
 and that the base norm for this new base is still $p_C$.

 Thus we simply add to our definition of base norm space above  the requirement (3): $X_+$ is closed. That is,
 a {\em real base norm space} is a pre-base norm space for which $X_+$  is closed.
One may equivalently define
a real base norm space to be an ordered normed space $(X, \| \cdot \|)$ with closed cone $X_+$ containing a  convex set $K$ in Ball$(X) \cap 
 X_+$ such that for every nonzero $x \in  X_+$ there 
 is a unique scalar $t \geq 0$ and unique $k \in K$  with $x = tk$, and such that    Ball$(X)  \subseteq t \, 
{\rm co} (K \cup (-K) )$ for all $t > 1$.    Indeed it is then easy to see that (1) and (2) in the definition of
a pre-base norm space hold, and the norm $\| \cdot \|$ agrees with $p_C$ above.

A special case of interest are the {\em dual base norm spaces}.  This is
a base norm space with a Banach space predual such that the base $K$ is 
weak* closed (and hence weak* compact).
If $X$ is a real dual base norm space 
then a simple compactness argument shows that  Ball$(X)  = 
{\rm co} ( K \cup -K )$.   In this case $f_1$ is weak* continuous.   (For 
suppose that we have a bounded net $x_t \to x$ weak*, and we wish to show that every convergent subnet of 
$f_1(x_t)$ converges to $f_1(x)$.  
Write $x_t = c_t k_t - d_t r_t$ for $k_t,  r_t \in K$, and positive $c_t, d_t$ with $(c_t + d_t)$ bounded.
Replacing the net by subnets we may assume that $c_t \to c, d_t \to d, k_t \to k, r_t \to r$, and $x = c k - d r$.
And $f_1(x_t) = c_t - d_t \to c-d$, so $f_1(x_t) \to f_1(x)$.)

 As we said in the introduction, it is well known that base norm spaces and aou spaces are in duality. 
Thus a normed ordered real vector space $F$ is an aou space (with norm agreeing with 
the order unit norm  if  and only if $F^*$ is a real dual  base norm space.
Thus the dual Banach space of a real unital function space is the generic  dual  base norm space.   Equivalently, the dual  base norm spaces are exactly (up to appropriate isomorphism) the spaces $A(K)^*$ for a compact convex set $K$.  The dual  base of $A(K)^*$ is $\delta(K)$, where $\delta : K \to A(K)^*$ is the canonical map. 

Similarly, a normed ordered real vector space $F$ with closed cone is a 
real   base norm space  if  and only if $F^*$ is a dual aou  space.  
Thus the generic  base norm space ``is" the predual of a dual aou space, with the base 
corresponding to the normal state space. See e.g.\ \cite[Corollary of Theorem 6]{Ellis} or \cite[Section 3.9]{Jam}.  For base spaces this is the analogue of Kadison's theorem. 

These facts will be reprised in the later sections.  We include a proof of the following little known result of Ellis \cite{Ellis2} for completeness, and because we will need to generalize it later.

 \begin{thm}
     \label{xpc} Let $(X,X_+)$ be a (real) pre-base norm space with base $K$, with base norm $\| \cdot \|$. Then the closure $\overline{X_+}$ of $X_+$ with respect to this norm is a proper cone in the completion $\bar{X}$.
   Moreover $(\bar{X},\overline{X_+})$ is a base norm space
which has the norm closure $\bar{K}$ of $K$ as a base, and which still has $\| \cdot \|$ as its 
 base norm.  \end{thm}    

\begin{proof} 
Let $(X,X_+)$ be a pre-base norm space, with its base norm.
Define $f \in X^*$ to be positive  if and only if   $f \geq 0$ on $X_+$. 
Then the usual classical argument (for the fact mentioned above that the dual of a base norm space is an aou space) shows that $X^*$ is an aou space
with order unit $f_1$.  Indeed this is an easy exercise
(see e.g.\ Lemma \ref{banbn} for the complex case of this).
Then $\bar{X}$ is a base norm space with the predual cone ${\mathfrak c}_*$ consisting of the $x \in \bar{X}$ with $f(x) \geq 0$ for all $f \in (X^*)_+$. The new base 
is the normal state space. 
(Again, see Lemma \ref{banbn} for the complex case of this.)
Then ${\mathfrak c}_* = \overline{X_+}$
by the geometric Hahn-Banach theorem or bipolar theorem (e.g.\ 3.1.7 in 
\cite{Jam}).  This is of course a proper cone. 
Next we show that $\bar{K}$ is a base for $\overline{X_+}$.  Suppose that $x \in \overline{X_+}$ with $\| x \| = 1$.  Let $x_n \in X_+$ with
$x_n \to x$.  Scaling, we may assume that $\| x_n \| = 1$ for each $n$, so that $x \in \bar{K}$.  That is, 
$\bar{K} = \{ x \in \overline{X_+} : \| x \| = 1 \}$.  Hence if  $s x = t y$ for 
$s, t > 0$ and $x, y \in \bar{K}$ then $s = t$, so that $x = y$.
From this it is easy to see that $\bar{K}$ is a base for $\overline{X_+}$.
\end{proof} 

We call this the {\em base-completion}. 
It follows from this (and its proof), and from elementary topological arguments, that if $X$ is an incomplete pre-base norm space then the closure of $X_+$ in $X$  is a proper cone which has the norm closure $\bar{K}$ of $K$ in $X$ as a base making $X$ a base norm space (with closed cone), 
 and that the base norm and base function for this new base is unchanged. 

\section{Complex base norm spaces} \label{cbns} 
Suppose that $X$ now is a  complex  $*$-vector space with cone 
$X_+ \subset X_{\rm sa}$
such that $(X_{\rm sa},X_+)$ is a real base norm space in the sense above, with base $K \subset X_+$.   We define the {\em extended base norm} or {\em  canonical base norm} on $X$ by \begin{equation}
    \label{l1n} 
    \vert \vert \vert u \vert \vert \vert = \inf  \sum_{k=1}^n \, | \alpha_k | , \end{equation} the infimum taken over all ways to write $u = \sum_{k=1}^n \, \alpha_k \, \varphi_k$ with $\varphi_k \in K, 
\alpha_k \in \bC$.  
We call $X$ together with this norm a {\em complex base norm  space}. 
On the selfadjoint part of $X$ the extended base norm is the usual base norm above.  Indeed for $u \in X_{\rm sa}$ the infimum in (\ref{l1n}) is easily seen to be  dominated  by the Minkowski norm 
$p_C(u)$.  To see the converse, 
suppose that $u = \sum_{k=1}^n \, \alpha_k \, \varphi_k$ with $\varphi_k \in K, 
\alpha_k \in \bC$ with $\sum_k \, |\alpha_k| < 1$.  Then 
$u = \sum_{k=1}^n \, {\rm Re}(\alpha_k) \, \varphi_k$, and it is easy to argue that this lies in $C$, 
so that $p_C(u) \leq 1$. 
The canonical extension (or complexification) of the base function $f_1$ for $X_{\rm sa}$ to $X$ is contractive.  
Indeed this is evident from the inequality 
$$\left|f_1(\sum_{k=1}^n \, \alpha_k \, \varphi_k)\right| = \left|\sum_{k=1}^n \, \alpha_k \right| \leq \sum_{k=1}^n \, |\alpha_k| .
$$  We sometimes simply write this extension as $f_1$, and call it  the {\em base function}.
Then $K$ lies in the associated hyperplane $H = f_1^{-1}(\{ 1 \})$. 

\begin{cor} \label{cubbn} The closed unit ball with respect to  the canonical base norm 
 of a complex base norm space is the closure of  the 
 absolute  convex hull of the base $K$. \end{cor}
 
It is easy to see that the norm in (\ref{l1n})  on 
a  complex  base norm space 
$E = E_{\rm sa} \oplus i E_{\rm sa}$ is 
the {\em dual Taylor norm} (or Bochnak norm) $\|\cdot \|_{T^*}$ on the complexification of $E_{\rm sa}$.  
 That is, it is the norm induced from the projective tensor product $E_{\rm sa} \, \hat{\otimes} \, l^2_2(\bR)$.  
Thus for $u \in E_{\rm sa} \oplus i E_{\rm sa}$, this norm is explicitly given by the formula $\|u\|_{T^*}=  \inf \, \sum_{k=1}^n \, | \alpha_k | \| \psi_k \|$, where the infimum is taken over all ways to write $u$ as $\sum_{k=1}^n \,  \alpha_k \otimes  \psi_k$ with $\alpha_k \in \bC, \psi_k \in E_{\rm sa}$. 
  Of course the Taylor and dual Taylor norms are in duality; 
indeed it is well known in Banach space theory that the dual of $X \, \hat{\otimes} \, l^2_2(\bR)$ is
 $X^* \, \overset{\vee}{\otimes} \, l^2_2(\bR)$, and similarly with $\hat{\otimes}$ and $\overset{\vee}{\otimes}$ switched.

\bigskip 

 {\bf Remark.} In \cite[Section 4.2]{PT} the authors define what they call the `maximal order seminorm' on an ordered $*$-vector space with order unit using a formula 
 similar to (\ref{l1n}) above, but beginning with the {\em order unit seminorm} on the selfadjoint part of the space. They show that this indeed defines a seminorm on the whole space and that it extends the existing norm on the selfadjoint part to the whole space. Their arguments do not make use of the order unit, and thus are applicable in our setting above
 and could be tweaked to obtain a few facts mentioned above.

\bigskip 

If $X$ is a complex base norm space
then $X$ has a Banach space predual (that is, 
$X$ is a dual  base norm space) if and only if  $X_{\rm sa}$ is a dual real base norm space.  (One direction of this follows by tensor duality because 
$X \cong X_{\rm sa} \, \hat{\otimes} \, l^2_2(\bR)$.   For the other, if $X$ has a Banach space predual $F$ then the space $F_{\rm sa}$ of selfadjoint weak* continuous functionals  is 
a Banach space predual of $X_{\rm sa}$ by Lemma \ref{sad}.)
In this case  we say that $X$ is a {\em complex dual base norm space}.  It follows that 
$f_1$ is weak* continuous.

The natural morphisms $u : (X,K_X) \to (Y,K_Y)$ between   base norm spaces (resp.\ dual  base norm spaces) we will call 
{\em base morphisms} (resp.\ {\em dual base morphisms}: namely  (selfadjoint)  positive   (resp.\  positive and weak* continuous) linear maps  between the base spaces which preserve the base, in the sense that $x \in K_X$ if and only if $u(x) \in K_Y$. 
This is equivalent to the base function of $Y$ composed with 
$u$ being the base function of $X$.  

\medskip

{\bf Example.}  The canonical example of such a dual base is the state space of a unital function system 
$V \subseteq C(K)$.  The main reason why this is a dual base is that if  $\psi \in ({\rm Ball}(V^*))_{\rm sa}$ then by the Hahn-Banach theorem $\psi$ extends to 
a contractive functional on $C(K)$.  By the Jordan decomposition this is a difference of two 
positive functionals whose norms sum to $\| \psi \| \leq 1$.  Thus $\psi = 
c_1 \psi_1 - c_2 \psi_2$ for states $\psi_i \in K$ and $c_i \geq 0$ with $c_1 + c_2 = \| \psi \|$. 
From this it is clear that ${\rm Ball}(V^*)_{\rm sa} = {\rm co} ( K \cup -K )$.
For the remaining details see the proof of Lemma \ref{banaou}.

\bigskip

{\bf Remarks.} 1)\
One may define a {\em complex base ordered space} to be a complex  $*$-vector space $F$ whose selfadjoint part is a real base norm space. 
In this case the base norm on $F_{\rm sa}$ extends to  a unique norm on $F$ with respect to which a  $F$ is a complex  base norm space, namely the dual Taylor norm.  A normed complex  $*$-vector space $F$ which is a complex base ordered space therefore has a canonical equivalent norm with respect to which it is a complex  base norm space.   This is analogous to the situation for aou spaces.    The predual of a von Neumann algebra is a base ordered space, but if we want it to be a base norm space it has to be equivalently renormed.  

\smallskip

2)\ The canonical `complex base norm' that we assigned to a complex base norm space $E$, namely the  dual Taylor norm, is the universal or biggest norm corresponding to the absolutely convex hull of $K$.  Namely it is  the normed complexification of $E_{\rm sa}$ which `contains the biggest-norm 
 absolutely convex hull' of $K$. 
 Indeed if $E_{\rm sa}$ is real isometrically embedded in a complex space $Y$,
 then by the property of the projective tensor product we obtain a contraction $$E \cong 
 E_{\rm sa} \, \hat{\otimes} \, l^2_2(\bR) \to Y \, \hat{\otimes} \,\bC \to Y$$ taking 
 $\psi \otimes \alpha$ to $\alpha \, \psi \in Y$, for $\psi \in E_{\rm sa}, \alpha \in l^2_2(\bR) = \bC$. 
 This contraction takes the absolutely convex hull of $K$ in $E$ onto the absolutely convex hull of $K$ in $Y$.  It also takes the complex (resp.\ real) span of $1 \otimes K$ onto the complex 
 (resp.\ real) span of $K$, which is $Y$ (resp.\ $Y_{\rm sa}$ if $Y$ is also a complex base norm space with base $K$, such as $Y = E$.  
 Moreover this contraction is one-to-one if $Y$ is a $*$-vector space with $K \subset Y_{\rm sa}$.  To see this suppose that $\sum_{k=1}^n \, \alpha_k \, \varphi_k = \sum_{k=1}^m \, \beta_k \, \psi_k$, for $\alpha_k, \beta_k \in \bC, \varphi_k, \psi_k \in K$.  Then $\sum_{k=1}^n \, {\rm Re}(\alpha_k) \, \varphi_k = \sum_{k=1}^m \, {\rm Re}(\beta_k) \, \psi_k$, and similarly with the imaginary parts.   Thus it follows that
 $$\sum_{k=1}^n \, \alpha_k \, \otimes \, \varphi_k = \sum_{k=1}^m \, \beta_k \,  \otimes \, \psi_k, $$ as is easily seen by writing any of the terms $\gamma  \, \otimes \, \xi$ here as $$(1 \cdot {\rm Re}(\gamma) + i {\rm Im}(\gamma))  \, \otimes \, \xi =  1 \, \otimes \, ({\rm Re}(\gamma)) \xi )
\, + \, i   \, \otimes \, ({\rm Im}(\gamma)) \xi ),$$ and then using linearity of $\otimes$ in the second variable to write both sides of the claimed equality in the form  $1 \, \otimes x + i \, \otimes y$.   Here $x$ will be $\sum_{k=1}^n \, {\rm Re}(\alpha_k) \, \varphi_k = \sum_{k=1}^m \, {\rm Re}(\beta_k) \, \psi_k$, and similarly with the imaginary parts. 

Thus we see that the absolutely convex hull of a base $K$ is defined uniquely, independent (up to affine isomorphism) of the containing $*$-vector space, since it is affine isomorphic to the absolutely convex hull of 
$1 \otimes K$ in the projective tensor product.

\begin{lem} \label{banaou} A normed ordered complex $*$-vector space $(F,F_+)$ is a 
complex archimedean order unit space (with norm agreeing with 
the order unit norm on $F_{\rm sa}$) if  and only if $F^*$ is a complex dual  base norm space. 
\end{lem} 

\begin{proof}
Suppose $F$ is a complex aou space.  Then the norm on $F$ is given by the extended order unit norm, or Taylor norm, with respect to $F_{\rm sa}$, as pointed out in the Introduction. Moreover, $F_{\rm sa}$ is evidently a real aou space, and so $(F_{\rm sa})^* \cong (F^*)_{\rm sa}$ is a real dual base norm space. We identify $(F^*)_{\rm sa} + i (F^*)_{\rm sa}$ with $(F_{\rm sa} + i F_{\rm sa})^*$ where $\varphi + i \psi$ in the former space is identified in the latter with the complex bounded linear functional  $x + iy \mapsto \varphi(x) - \psi(y) + i(\psi(x) + \varphi(y))$. Since $(F_{\rm sa}\overset{\vee}{\otimes} \, l^2_2(\bR))^* \cong (F_{\rm sa})^*\hat{\otimes} \, l^2_2(\bR)$, the identification above is an isometric isomorphism, so that $(F^*)_{\rm sa} + i (F^*)_{\rm sa}$ is a complex dual base norm space space.  Moreover, it shows that the norm on $F^*$ as a dual space agrees with the dual Taylor norm with respect to $(F^*)_{\rm sa}$. Thus the dual of a complex archimedean order unit space is a complex dual base norm space.  

For the converse, we assume that $F^*$ is a complex dual base norm space.  From the definition of a complex base norm space it is immediate $(F^*)_{\rm sa}$ is a real base norm space.  Since $(F^*)_{\rm sa} \cong (F_{\rm sa})^*$, it is a real dual base norm space. Hence, by the real theory, $F_{\rm sa}$ is an aou space,  and so $F = F_{\rm sa} + i F_{\rm sa}$ with the extended order unit norm (the Taylor norm) is a complex archimedean order unit space.  However this is
precisely the original norm on $F$ by duality, since
 the norm 
on $F^* = (F^*)_{\rm sa} + i (F^*)_{\rm sa}$ is the dual Taylor norm.  
\end{proof}

{\bf Example.} Thus the dual Banach space of a complex function  system is the generic  dual  base norm space.   Equivalently, the dual  base norm spaces are exactly (up to appropriate isomorphism) the spaces $A(K)^*$ for a compact convex set $K$.  The dual  base of $A(K)^*$ is $\delta(K)$, where $\delta : K \to A(K)^*$ is the canonical map. 

From the next result we see that similarly the generic  base norm space ``is" the predual of a dual aou space, with the base 
corresponding to the normal state space. For base spaces this is the analogue of Kadison's theorem. 

\begin{lem} \label{banbn} A 
complete
normed ordered complex $*$-vector space $(F,F_+)$ with $F_+$ closed 
is a 
complex  base norm space  if  and only if $F^*$ is a dual complex archimedean  order unit  space.  The normal state space of the latter corresponds to the base of $F$.
\end{lem} 

\begin{proof} Suppose that $E$ is a unital complex function system with a Banach space predual $F$, and $K$ is the `normal state space' in $F$.   It is well known that $F_{\rm sa}$ is a real  base norm space \cite{AE}, however for the readers convenience we give a proof.  We may assume that $E$ is a weak* closed subsystem of a commutative von Neumann algebra $M$ (for example \cite[Corollary 2.2]{BMag} shows this with $M= l^\infty(I)$). 
Note that every selfadjoint weak* continuous functional on $E$ extends to a selfadjoint weak* continuous functional 
$\psi$ on $M$
with close norm.   This is because $E_* \cong M_*/E_\perp$, so that a selfadjoint weak* continuous functional on $E$ extends to a weak* continuous functional $\rho$ on $M$
with close norm. Then $\rho^*$ is also such an extension, so that $\psi= \frac{1}{2}(\rho +\rho^*)$ is a selfadjoint extension
with close norm.  We may write $\rho= c_1 \varphi_1 - c_2 \varphi_2$ with $c_i \geq 0, c_1 + c_2 = \| \rho \|$, and for normal states $\varphi_i$ on $M$ and hence on $E$.  We are using the fact that every positive (weak* continuous) functional on $E$ is a (unique) positive scalar multiple of a (weak* continuous) state.  The hyperplane in $F$ containing $K$  is the one defined by $\psi(1) = 1$. In other words, $f_1$ is evaluation at $1$. 

The norm on $F = F_{\rm sa} + i F_{\rm sa}$ is the 
dual Taylor norm on the complexification of $F_{\rm sa}$, that is, $F_{\rm sa} \, \hat{\otimes} \, l^2_2(\bR)$.  This follows by duality as in the last theorem, since
the norm on $E$ is the Taylor norm.   Thus $F$ is a complex base norm space.  

Next suppose that  $X$ is a complex base norm space.
So $X_{\rm sa}$ is a real base norm space, hence its dual is an aou space.  Thus there exists a convex compact $K$ and 
surjective unital isometry $\rho :  (X_{\rm sa})^* \to A(K,\bR) = A(K,\bC)_{\rm sa}$.  The order unit 
in $(X_{\rm sa})^*$ is a positive (hence selfadjoint) functional on $X$ which is $1$ on the base \cite{AE}. 
There is a unique such functional as we indicated before.  
The norm on $X$ is the `dual Taylor norm' complexification of $X_{\rm sa}$.  Thus the norm on $X^*$ is the
Taylor norm complexification of $(X^*)_{\rm sa}
\cong (X_{\rm sa})^*$.  And the norm on $A(K,\bC)$ is  the
Taylor norm complexification of $A(K,\bR)$. Thus 
the unital isometry $\rho$ extends to a unital surjective selfadjoint isometry $X^* \to A(K,\bC)$. 
   So $X^*$ is a complex aou space.   
\end{proof}

 {\bf Remark.} Alfsen has a definition of base norm spaces in \cite{Alfsen} which is purely order theoretic.  However his definition and result is somewhat restrictive since many interesting base norm spaces in our sense (the sense of  \cite{AE})  are not base norm spaces in Alfsen's sense.  The normal state space of a  von Neumann algebra $M$ does however satisfy his definition on the selfadjoint part.  However this very example is not a complex base norm space in our sense except if $M$ is commutative, nor is its dual a complex normed aou space ($M$ is not a complex function space).

 \bigskip

 The following characterization of dual base norm spaces  is rather trivial (it essentially asserts that $E_*$ is isometric to $A(K)$), however  it is one way to get around 
 the issues we exposited earlier concerning the norm on a complex base norm space.  There is a nc variant of this which we omit  in view of its triviality. 
 
\begin{prop}  Suppose that  $E$ is the Banach space dual of an ordered operator space and $*$-vector space $F$, with the canonical dual ordering. Then $E$ is an ordered $*$-vector space with the canonical dual ordering.  
Suppose that $K$ is a compact convex set in 
the selfadjoint and positive part of ${\rm Ball}(E)$, such that $E = {\rm Span} \, (K)$.   We also assume that every $f \in A(K)$ has a norm preserving linear extension to a  weak* continuous functional on $E$.
 Then $E$ is a dual base norm space with base $K$, and  
 we have a selfadjoint surjective  isometric isomorphism $A(K)^* \cong E$ taking $\delta(k)$ to $k$ for $k \in K$.  \end{prop}

\section{Noncommutative base norm spaces} \label{ncc}

For a nc set $K$ in  a matrix ordered space, our nc version of {\em every positive $x$ is a positive scalar multiple of an element of the base} $K$, is that for each $n \in \bN$, 
every $u \in M_n(E)_+$ is of form $\alpha^* k \alpha$ for $k \in K_n, \alpha \in M_n$, indeed with $\alpha$ positive, since because $n$ is finite, by the polar decomposition 
$\alpha^* k \alpha = |\alpha | U^* k U |\alpha | = |\alpha | k' |\alpha |$ for $k' \in K$.   If this condition holds we shall say that $E$ is {\em based on} $K$.    
It turns out that it is not necessary to usually mention this in the definition below, because it follows from other conditions below.

\begin{definition} \label{ncbns}  Let $E$ be a (real or complex) matrix ordered matrix normed $*$-vector space
(in particular, recall that the matrix cones $M_n(E)_+$ are closed, and that $\| x^* \|_n = \| x \|_n$ for $x \in M_n(X)$).  We will also assume that $E$ is an operator space
(although the later result Theorem \ref{doesd}
shows that it is not necessary to say this, since it follows from conditions 
like the ones that follow). 
Let 
$K$ be a  
(real or complex) matrix convex set in $E$.  Suppose that 1)\ $K_n \subset {\rm Ball}(M_n(E))_+$ for all $n$.  2)\ $K$ is closed at each level.
3)\ For every $t > 1$, every element in Ball$(M_n(E))_{\rm sa}$ is of form
 $x = c_1 x_1 c_1 - c_2 x_2 c_2$ with $c_i$ positive matrices, with $\| c_1^2 + c_2^2 \| \leq t$ (or equivalently, 
 $c_1^2 + c_2^2  \leq t I$) and $x_i \in K_n$.   We think of this as a nc version of the condition Ball$(E)_{\rm sa}  \subseteq  
t \, {\rm co} (K \cup (-K))$, which we had in the classical case above.
Then, 4)\ we assume that $K$ lies in a closed nc hyperplane $H$ not passing through 0, and indeed that $K_n = M_n(E)_+ \cap H_n$ for each $n$.  Indeed we assume that  $H_n = \{ x \in M_n(E) : (f_1)^{(n)}(x) = I_n \}$ for a fixed scalar valued functional $f_1$ on $E$ 
which is positive (positivity is automatic if $E$ is based on $K$).  In the real case we insist that this functional is selfadjoint (in the complex case this is automatic).  By a standard argument
   $f_1$ is completely positive if it is positive. 
 If  1)--4) hold then we say that $E$ is a {\em nc base norm space}. 
\end{definition}

If a matrix ordered space $E$ is based on $K$ and condition 4) above holds then we say that $K$ is a {\em nc base} for $E$, and that $f_1$ in 4) is  the {\em base function}.

 We shall see that shortly that every nc base norm space is based on $K$, so that 3) may be rewritten as:
 $$\| x \|_n = \inf \{ \| (f_1)^{(n)}(y+z) \| : x = y-z, y, z \in M_n(E)_+ \} , \qquad x \in M_n(E)_{\rm sa} .$$  This also uses that  $(f_1)^{(n)}(c_i x_i c_i) = c_i^2$ for $x \in K_n$. 
  It is then easy to see that the base  function is unique, since it is $I$ on $K$ and  $E$ being based on $K$ and 3) imply that $E$ is a noncommutative span ncSpan$(K)$ of $K$ (see \cite[Section 3.1]{Breg}).

\bigskip

{\bf Remark.}  
In the real case unless  the involution on $E$ is trivial (at level 1), that is, it is the identity map,
then the base 
$K_1$ may not span $E$. However as we just said, $E = {\rm ncSpan}(K)$, using  e.g.\ 3) above and the fact that $x = [ I \; 0 ] \tilde{x} [ 0 \; I ]^\tran$. 
Duals of real $C^*$-algebras such as the quaternions or 
$\bC$ (as a real $C^*$-algebra) are particularly interesting here 
because the selfadjoint part may be trivial at level 1.

\begin{lemma} \label{ts}  Suppose that $E$ is a matrix ordered space with nc base $K$ with $K_n \subset M_n(E)_+$  for all $n \in \bN$.  Fix $n \in \bN$.  \begin{enumerate} 
            \item [{\rm (1)}]  Suppose $x \in M_n(E)_{+}$.  Then there exists $t > 0$ and $k \in K_n$ with $x \leq t k$ if and only if   $x = \alpha^* k \alpha$ for $k \in K_m, \alpha \in M_{m,n}$.  One may take $m = n$ and $\alpha = (f_1)^{(n)}(x)^{\frac{1}{2}}$, and in the 
            converse direction one may take $t = \| \alpha \|^2 = \| (f_1)^{(n)} (x) \|$.
  \item [{\rm (2)}]   For all $x \in M_n(E)_{\rm sa}$ there exists $t > 0$ and $k \in K_n$ with $x \leq t k$,  if and only if   $E$ is based on $K$ and  $M_n(E)_{\rm sa} = M_n(E)_+ - M_n(E)_+$. 
  \item [{\rm (3)}]   Definition {\rm \ref{ncbns}}  implies that $E$ is based on $K$, indeed $K$ is a  nc base for $E$.   \end{enumerate} 
\end{lemma}
\begin{proof}  (1)\  ($\Rightarrow$)\ Suppose that $x \in M_n(E)_{+}$ and $x \leq tk$ for $t > 0$ and $k \in K_n$.  By scaling we may assume $t = 1$.  Then $0 \leq a = f_1^{(n)}(x) \leq I$. Let $e$ be the support projection of $a$, and let $z = ((e a e)^{-\frac{1}{2}} + e^{\perp}) x  ((e a e)^{-\frac{1}{2}} + e^{\perp})  + e^{\perp} k e^{\perp}$.  Then $$f_1^{(n)}(z) = ((e a e)^{-\frac{1}{2}} + e^{\perp}) a ((e a e)^{-\frac{1}{2}} + e^{\perp}) + e^{\perp} = e + e^{\perp} =  I_n.$$
Thus $z \in K_n$ and $x = a^{\frac{1}{2}} z  a^{\frac{1}{2}}$.   

 ($\Leftarrow$)\  By scaling we may assume $\| \alpha \| \leq 1$.  Then for any $k' \in K_n$ we have by nc convexity that 
 $$\alpha^* k \alpha \leq  \alpha^* k \alpha  + (I-\alpha^* \alpha)^{\frac{1}{2}} k_2  (I-\alpha^* \alpha)^{\frac{1}{2}} \in K_n.$$ 

(2)\  ($\Rightarrow$)\  This follows easily from (1).  Note if $x \leq t k$ then $x = tk - (tk-x)$. 

($\Leftarrow$)\   If $x \in M_n(E)_{\rm sa}$  then $x$ is dominated by a positive element, hence by an element of form $\alpha^* k \alpha$ since  $E$ is based on $K$.  Now apply (1). 

(3)\ Let $x \in M_n(E)_+$.  From Condition  3) we have $x \leq c k c$ for $k \in K_n, c \in M_{n}$.  By (1) above used twice, $x \leq c k c \leq t k'$ for some $k' \in K$, so that $x = \alpha^* k'' \alpha$ for 
$k'' \in K_n, \alpha \in M_{n}$. 
\end{proof}  

{\bf Remark.} We used that $n$ is finite to ensure that $(e a e)^{-\frac{1}{2}}$ is bounded, so that $z$ is well defined. 

\begin{lemma} \label{bn}
If $E$ is a nc base norm space and $n \in \bN$, then for selfadjoint $x \in M_n(E)$,
\begin{equation} 
   \label{sa} \|x\|_n = \inf\{\|\alpha_1^*\alpha_1 + \alpha_2^*\alpha_2\|: \alpha_i \in M_{n}, x_i \in K_n, x = \alpha_1^*x_1\alpha_1 - \alpha_2^* x_2 \alpha_2\}. \end{equation}
Also, if $x$ is positive, then $\| x \|_n = \| f_1^{(n)}(x) \|$ for $x \in M_n(E)_+$.  Moreover $K_n = 
\{ x \in M_n(E)_+ : f_1^{(n)}(x) = I_n \}$.
\end{lemma}
\begin{proof}  Eq.\  (\ref{sa})  follows from the displayed equation before the Remark above Lemma \ref{ts}.  Note 
 if $x = \alpha_1^*x_1\alpha_1 - \alpha_2^* x_2 \alpha_2$ with $\alpha_i \in M_{n}$, and $x_i \in K_n$, then
$$\| x \|_n = \| [\alpha_1^* \; \alpha_2^* ] (x_1 \oplus (-x_2))  [\alpha_1 \; \alpha_2 ]^\tran \| \le \|\alpha_1^*\alpha_1 + \alpha_2^*\alpha_2\| .$$
(A similar argument, if necessary, gives $\| \alpha_1^*\alpha_1 - \alpha_2^*\alpha_2  \| \leq \| \alpha_1^* \alpha_1 + \alpha_2^*\alpha_2 \|$.) 

For the second part, condition 4) above implies that $f_1$ satisfies 
 $K_1 = E_+ \cap f_1^{-1}(\{ 1 \})$.  Then $f_1$ is selfadjoint and positive, and is strictly positive since  $E$ is based on $K$. So, certainly $f_1$ is completely contractive, being a contractive positive functional.  For the reverse inequality, note that if $x \in M_n(E)$ is positive, then $x = ckc$ for $k \in K_n$ and a positive matrix $c$.  Then $\|f_1^{(n)}(x)\| = \|c^2\| \ge \|x\|$, where the last inequality follows from the first part of the lemma.
The last assertion is clear from $K_n = M_n(E)_+ \cap H_n$.
\end{proof} 

{\bf Remark.} That $\| x \| = \| f_1^{(n)}(x) \|$ above can also be viewed as the fact that the norm of a completely positive map on an operator system is the norm of its value at $1$.

\medskip

It is important that the norm in (\ref{sa}) may be written in a couple of different ways.
We have 
$$\|x\|_n = \inf\{\|\alpha_1^*\alpha_1 + \alpha_2^*\alpha_2\|: \alpha_i \in M_{m,n}, x_i \in K_m, x = \alpha_1^*x_1\alpha_1 - \alpha_2^* x_2 \alpha_2, n \leq m \}.$$ 
Indeed this follows from the polar decomposition trick  mentioned at the start of this section, to write $\alpha_i = U_i |\alpha_i|$ with an {\em isometry} $U_i$, and 
using that $U_i^* K_k U_i \subseteq K_n$. 
The norm of a general $x \in M_n(E)$ may be phrased in terms of the norm of a selfadjoint matrix by the usual trick: \begin{equation}
    \label{gen} \|x\|_n = \left\|\left[\begin{array}{cc} 0 & x \\ x^* & 0\end{array}\right]\right\|_{2n}.
\end{equation} 

As in the `classical case' considered in earlier sections, it will also be important to show how a (real or complex) matrix ordered  $*$-vector space $E$ with 
a matrix convex set $K$ (but no norm as yet) may be made into a nc base norm space.  We say that $E$ is a  {\em matrix base ordered space} if 
\begin{enumerate} 
            \item [{\rm (a)}]  We assume that as in 4) of Definition \ref{ncbns},  there is a strictly positive (i.e.\ faithful) 
selfadjoint functional $f_1$ on $E$ such that $K_n = M_n(E)_+ \cap H_n$,
where $H_n = \{ x \in M_n(E) : (f_1)^{(n)}(x) = I_n \}$.  We call $f_1$ the {\em base function} for $K$.  It follows as usual that $f_1$ is completely positive.
   \item [{\rm (b)}]  For every $n \in \bN$ and $x \in M_n(E)_{\rm sa}$,  there exists $t > 0$ and $k \in K_n$ with $x \leq t k$.     \item [{\rm (c)}]   If $x \in M_n(E)_{\rm sa}$ and for 
every  $\epsilon > 0$ we can write $x = y-z$ for $y, z \in M_n(E)_+$ with $f_1^{(n)}(y)$ and $(f_1)^{(n)}(z)$ both of norm $< \epsilon$, then $x = 0$.
\end{enumerate} 
Note that the condition in (b) is equivalent, by Lemma \ref{ts}, to:  $E$ being based on $K$ and $M_n(E)_{\rm sa} = M_n(E)_+ - M_n(E)_+$ 
   for each $n \in \bN$.
We then {\em define} $\| \cdot \|_n$ on 
$M_n(E)_{\rm sa}$ by 
$$\| x \|_n = \inf \{ \| (f_1)^{(n)}(y+z) \| : x = y-z, y, z \in M_n(E)_+ \} , \qquad x \in M_n(E)_{\rm sa} .$$ 
By 
Lemma \ref{ts}, this agrees with the equation  in (\ref{sa}), since $(f_1)^{(n)}(\alpha^* k \alpha) = \alpha^*\alpha$. We then 
define $\| \cdot \|_n$ on all of 
$M_n(E)$ by Equation (\ref{gen}). We will show later that this is well defined.

\begin{theorem} \label{doesd} The expressions $\| \cdot \|_n$
just defined on a matrix base ordered space $E$ are matrix norms with respect to which $E$ is
an operator space and matrix ordered matrix normed $*$-vector space,
and $K$ lies in 
the positive part of the matrix unit ball of $E$.
Also, $f_1$ is (completely) contractive on $E$, and 
$\| x \| = \| f_1^{(n)}(x) \|$ for $x \in M_n(X)_+$. 
Indeed $E$ satisfies all of the conditions to be a nc base norm space except for possibly $M_n(X)_+$ and $K_n$ being closed in the norm topology for $n \in \bN$.  \end{theorem} 

\begin{proof} We first show  that $\| \cdot \|_n$ as defined  in 
(\ref{sa}) is a  norm. 
By definition if $x \in M_n(X)_{\rm sa}$ and 
$\| x \|_n = 0$ then $x = 0$.  If $x = \alpha_1^*x_1\alpha_1 - \alpha_2 x_2 \alpha_2$ and $y = \beta_1^*y_1\beta_1 - \beta_2 y_2 \beta_2$, then $$x + y = (\alpha_1^* x_1 \alpha_1 + \beta_1^*y_1\beta_1) - (\alpha_2^* x_2 \alpha_2 + 
\beta^*_2 y_2 \beta_2) =  \gamma_1^* z_1 \gamma_1 - \gamma_2^* z_2 \gamma_2$$ where $\gamma_i^* = [\alpha_i^*  \; \beta_i^* ]$ and $z_i = x_i \oplus y_i \in K$.  Note that 
$$\gamma_1^* \gamma_1 + \gamma_2^* \gamma_2 = 
(\alpha_1^* \alpha_1 + \alpha_2^* \alpha_2) + (\beta_1^* \beta_1 + \beta_2^* \beta_2),$$
from which the triangle inequality holds.  
Note that if $x = \alpha_1^*x_1\alpha_1 - \alpha_2 x_2 \alpha_2 \in M_n(E)_{\rm sa}$ and $\beta \in M_n$ then \begin{equation}\label{saruan1}\| \beta^* x \beta \|  \le  \|(\alpha_1\beta)^*(\alpha_1\beta) + (\alpha_2\beta)^*(\alpha_2\beta)\|
     \le \|\beta^*\| \|\alpha_1^*\alpha_1 + \alpha_2^*\alpha_2\| \|\beta\|.\end{equation}
     Hence $\| \beta^* x \beta \|  \le \|\beta^*\|\|x\|\|\beta\|.$   
So our infimum expression  is a (real) norm on the selfadjoint part. Since 
$$(f_1)^{(n)}(x) = \alpha_1^*(f_1)^{(n)}(x_1)\alpha_1 - \alpha_2 (f_1)^{(n)}(x_2) \alpha_2 = \alpha_1^*\alpha_1 + \alpha_2^*\alpha_2,$$
it is clear that $f_1$ is (completely) contractive on the selfadjoint part. 
From this it is easy to see
that (\ref{gen}) also defines a norm on $M_n(E)$.
One part of this follows for example from 
$\| \beta^* \tilde{x} \beta \|  \le \|\beta^*\|\| \tilde{x} \|\|\beta\|$, with  $\beta = (e^{-i \theta} I) \oplus I$.
It is also easy to see that $f_1$ is (completely) contractive in this norm on $E$.  That $\| x \| = \| f_1^{(n)}(x) \|$ for $x \in M_n(X)_+$ then follows as in Lemma \ref{bn}. 

We check that this definition satisfies Ruan's axioms for the matrix norms of an operator space, which will have as a byproduct that the norms in (\ref{sa}) and (\ref{gen}) coincide on $M_n(E)_{\rm sa}.$   For the second Ruan axiom, suppose that $x \in M_m(E)$ and $y \in M_n(E)$ are both selfadjoint.  
From the equation \begin{equation*} \begin{split} x \oplus y & = (\alpha_1^*x_1\alpha_1 - \alpha_2 x_2 \alpha_2) \oplus (\beta_1^*y_1\beta_1 - \beta_2 y_2 \beta_2) \\ & = (\alpha_1  \oplus \beta_2)^*(x_1 \oplus y_1) (\alpha_1  \oplus \beta_1) - (\alpha_2  \oplus \beta_2)^*(x_2 \oplus y_2) (\alpha_2  \oplus \beta_2), \end{split} \end{equation*}  it is readily shown that $\|x \oplus y \|_{m+n} \le \max\{\|x\|_m, \|y\|_n\}$.  To get the opposite inequality, note that if $x \oplus y = \alpha_1^*x_1\alpha_1 - \alpha_2 x_2 \alpha_2$ then $x = \beta^* \alpha_1^* x_1\alpha_1 \beta - \beta^* \alpha_2^* x_2\alpha_2\beta$ where $\beta^* = [I \;  0]$.
Thus  $$\| x \| \leq \| \beta^* \alpha_1^* \alpha_1 \beta + \beta^* \alpha_2^* \alpha_2 \beta \| \leq \| \alpha_1^* \alpha_1  + \alpha_2^* \alpha_2 \|,$$ so that $\| x \|_m \leq \|x \oplus y \|_{m+n}$, and similarly for $y$. 
For general $x \in M_m(E)$ and $y \in M_n(E)$, let $v = x \oplus y$. Then by a canonical shuffle
$$\| v \| = \| \tilde{v} \| = \| \tilde{x} \oplus \tilde{y} \|
= \max \{ \| \tilde{x} \| , \| \tilde{y} \| \}
= \max \{ \| x \| , \| y \| \}.$$

Using (\ref{saruan1}) above, for a general element $x \in M_n(E)$, and $\alpha, \beta \in {\rm Ball}(M_n)$, we have 
$$\left\|\left[\begin{array}{cc} 0 & \alpha x \beta \\ (\alpha x \beta)^* & 0\end{array}\right]\right\| = \left\|\left[\begin{array}{cc} \alpha & 0 \\ 0 & \beta^* \end{array}\right] \left[\begin{array}{cc} 0 & x \\ x^* & 0\end{array}\right] \left[\begin{array}{cc} \alpha^* & 0 \\ 0 & \beta \end{array}\right]\right\|
\le \| \tilde{x} \|_{2n}.$$ 
From this it is easy to see the first axiom of Ruan holds on $E$. 
Also, by these axioms and a canonical shuffle, 
$\| x^* \| = \| \widetilde{x^*}  \| = \| \tilde{x} \|= \| x \|$ for $x \in M_n(E)$.
If $x = x^*$ then by Ruan's axioms the norms in (\ref{sa}) and (\ref{gen}) coincide.

Since $K_n = M_n(E)_+ \cap (f_1^{(n)})^{-1}(I_n)$, 
if $M_n(E)_+$ is closed then so is $K_n$. 
Of course $K$ is contained in the closed hyperplane defined by $f_1$. 
\end{proof}

\begin{cor} \label{doesdc} Let $(E, \{ \| \cdot \|_n \})$ be a matrix ordered  matrix normed $*$-vector space with a matrix convex set $K$ in 
the positive part of the matrix unit ball of $E$.  Then $E$ is a nc base norm space if and only if  $E$ is a matrix base ordered 
space whose base matrix norms agree with $(\| \cdot \|_n)$
(or equivalently, $(\| \cdot \|_n)$ satisfies {\rm 3)} in the definition of a nc base norm space, as well as Eq.\ {\rm (\ref{sa})}), and $M_n(X)_+$ and $K_n$ are closed in the norm topology for $n \in \bN$.  \end{cor} 

\begin{proof} The one direction is obvious. If the conditions after the `if and only if' hold, then by Theorem \ref{doesd}, 
$E$ is an operator space, and indeed is a nc base norm space.  
Note that if {\rm 3)} in Definition \ref{ncbns} holds, then the  base matrix norms agree with $(\| \cdot \|_n)$ on selfadjoint matrices.  Thus they agree on all matrices if in addition  {\rm (\ref{sa})} 
holds for $(\| \cdot \|_n)$.   \end{proof}

\begin{definition} \label{dncbns} Suppose that  $E$ is the operator space dual of a 
matrix ordered matrix normed  
operator space and  $*$-vector space $F$, with the canonical dual matrix ordering. Then $E$ is a dual operator space and matrix ordered $*$-vector space with the canonical dual ordering.  
A {\em matrix dual base} (resp.\ {\em nc dual base} for $E$ is a matrix convex  set 
(resp.\ nc convex set in the sense of \cite{DK}) $K$ in 
the selfadjoint and positive part of the matrix unit ball of $E$, such that: 
1)\ $K$ is compact (at each level). 
2)\ Every $u \in M_n(E)_+$ is of form $\alpha^* k \alpha$ for $k \in K_n, \alpha \in M_n$, indeed with $\alpha$ positive, as above in the non-dual case. So $E$ is based on $K$. 
3)\ 
For all $n$, 
 Ball$(M_n(E))_{\rm sa}$ consists of the expressions $x = c_1 x_1 c_1 - c_2 x_2 c_2$ with $c_i$ positive  matrices with $c_1^2 + c_2^2  \leq I$ (or equivalently, $\| c_1^2 + c_2^2 \| \leq 1$) and $x_i \in K_n$.   
Again, we think of this as a nc version of the condition Ball$(E)_{\rm sa}  = 
{\rm co}( K \cup (-K))$.
Note that $\| x \| \leq \| c_1^2 + c_2^2 \|$ as in  the calculation in the proof of Lemma \ref{bn}.
Finally, as before, 4)\ we assume that $K$ lies in a  nc hyperplane not passing through 0, indeed that $K_n = M_n(X)_+ \cap H_n$ for each $n$, but we insist that the hyperplane is weak* closed, so that the functional $f_1$ defining the hyperplane is weak* continuous.   In the real case we also assume that $f_1$ is selfadjoint.  
If  1)--4) hold  
then we say that $E$ is a {\em matrix dual base norm space} (resp.\ {\em nc dual base norm space}).

We saw that 2) follows from 3) and 4) if $n \in \bN$ if $f_1$ is positive.  Thus 2)\ is automatic in the matrix convex  set 
case  if $f_1$ is positive.   It is similarly automatic (including at infinite levels) in the nc convex case.  This follows from Theorem \ref{banaou2} below and its proof, which only uses the finite $n$ version of the definitions above. 
We are also using the last assertion in the Introduction.  
\end{definition}

{\bf Remarks.}  1)\ In the definition above of a {\em dual base} we used either a matrix convex set or a nc convex set $K$.   
We will show later that these are equivalent; every matrix dual base can be canonically augmented to be a nc dual base.  For nc bases in our earlier sense that are not nc dual bases there are reasons to avoid
 (or be more careful with) infinite levels (see the proof of Theorem \ref{banbn2}). 
 
 \medskip
 
 2)\ It is not hard to argue from ideas mentioned above that  $E$ in the definition is a nc dual base norm space (that is, satisfies 1)--4))  if and only if it is a nc  base norm space
 and $K$ is weak* compact.  A similar assertion holds for  matrix dual base norm spaces.   We omit the proof since it proceeds via the proof of the later result Theorem \ref{banaou2}.

\medskip

As in Section 3, the natural morphisms $u : (X,K_X) \to (Y,K_Y)$ between  nc base norm spaces (resp.\ nc dual  base norm spaces) we will call 
{\em nc base morphisms}: namely  (selfadjoint) completely positive   (resp.\ completely  positive and weak* continuous) linear maps  between the nc base spaces which preserve the base, in the sense that $x \in K_X$ if and only if $u(x) \in K_Y$. 
This is equivalent to the nc base function of $Y$ composed with 
$u$ being the base function of $X$.   In the matrix case these are exactly what are called {\em quantum channels} by physicists, or CPTP maps ({\em completely positive trace preserving}). 

\begin{lem} \label{isa} A nc dual base norm space is a 
matrix dual base norm space, and is a nc base norm space.  If $X$ is a complex nc base norm space 
with base $K$ then $X$ is a base ordered space in the sense of Remark 1 before Lemma {\rm \ref{banaou}}, with base $K_1$. 
Indeed the nc base norm at level 1 on $X_{\rm sa}$ is exactly the base norm induced by $K_1$  on $X_{\rm sa}$.   Similarly the selfadjoint part of a nc (or matrix) dual base norm space is a real dual base norm space. 
\end{lem} 

\begin{proof}  For the last assertion,  the selfadjoint part of a nc dual base norm space $X$ is a dual space.  Indeed if $X$ is complex and has a Banach space predual $F$ then the space $F_{\rm sa}$ of selfadjoint weak* continuous functionals  is 
a Banach space predual of $X_{\rm sa}$ by Lemma \ref{sad}.  The rest is clear. 
 \end{proof} 
 
 The last result  is analogous to the aou space case, where an operator system is also an aou space, with order norm agreeing with the operator system norm on $X_{\rm sa}$.  

\bigskip
 
{\bf Remark.} Unlike the classical case, the matrix norms on the positive cones of a nc base norm space are not additive.  For example if $u_{ii}(x) = x_{ii} I_2$
for $x \in X = M_2$, then viewing $u_{ii} \in M_2(X^*)_+$ we do not have 
$\| u_{11} + u_{22} \| = \| u_{11} \| + \| u_{22} \|$.

The predual of a unital von Neumann algebra is a complex  nc base norm space.  More generally we have:

\begin{thm} \label{banbn2} A complete 
matrix ordered matrix normed complex  (resp.\ real) $*$-vector space $F$ is a 
complex (resp.\ real)  nc base norm space  if  and only if $F^*$  is a dual complex (resp.\ real) operator system.  Equivalently, if  and only if 
$F^*$  is a matrix order unit space with order unit coming from the nc hyperplane containing the nc base of $F$.  Moreover, this can be done with the dual operator space matrix norms for $F^*$ agreeing with 
the matrix order unit norms.
\end{thm} 

\begin{proof}  First assume we are in the complex case.
If $V$ is an operator system which has predual $F$ then $F$ is
a matrix ordered matrix normed  operator space with closed matrix cones.
The noncommutative normal state space $K$ is a nc  base
for $F$ as we now check. Fix $n \in \bN$. 
Clearly $K_n = {\rm UCP}^\sigma(V,M_n) = M_n(F)_+ \cap H_n$ where $H$ is the obvious weak* closed hyperplane defined by $1 \in V$, and UCP$^\sigma(V,M_n)$ is the normal matrix  state space of $V$.  
By \cite[Lemma 2.16]{BP} every $u \in M_n(V^*)_+$ is of the form $\alpha^* k \alpha$ for $k \in K_n, \alpha \in M_n$, indeed 
with $\alpha$ positive.  Combining this with the normal version of the Wittstock decomposition, Ball$(M_n(F))_{\rm sa}$ consists of the expressions $x = c_1 x_1 c_1 - c_2 x_2 c_2$ with $c_i$ positive contractive matrices with $\| c_1^2 + c_2^2 \| \leq 1$ and $x_i \in K_m$. We sketch a proof of this normal version of the Wittstock decomposition: Indeed if we suppose that $V$ is a weak* closed subsystem of $M = B(H)$, and $t > 1$, then every selfadjoint weak* continuous completely contractive
$u : V \to M_n$ has a weak* continuous completely bounded extension 
$\tilde{u} : M \to M_n$ of norm $< t$.  This is simply the fact that 
$V_*$ being a complete quotient of $M_*$ means that $M_n(V_*)$ is a quotient of $M_n(M_*)$.  We are assuming $n < \infty$ here
(the assertion is false if $n$ is infinite). 
By averaging with $\tilde{u}^*$ it is selfadjoint.
Any selfadjoint weak* continuous completely bounded map $v$ from $M$ into $M_n$ is a difference $u_1 - u_2$
of two weak* continuous completely positive maps with $\| u_1 + u_2 \| = \| v \|$.  This follows by a modification of the proof of the Wittstock Jordan theorem  in \cite[Theorem 8.5]{Pnbook}, but using the dual version of the Wittstock-Stinespring representation.  That is, first write $v = V^* \pi(\cdot) W$ for Hilbert space contractions $V, W$ and a normal $*$-representation $\pi$ of $M$ (see e.g.\ \cite[Theorem 2.7.10]{BLM}).
Then follow the proof in \cite[Theorem 8.5]{Pnbook}. 
 
 Next suppose that  $X$ is a complex nc base norm space with nc base $K$.  Then $X^*$ is certainly a complex matrix ordered space.  Let $e$ be the positive functional on $X$ corresponding to the nc hyperplane containing $K$.  This is positive  since $e$ is 1 on $K_1$ hence positive on $X_+$.
 Claim: $(X^*,e)$ is an archimedean matrix ordered space.  To see that $e$ is a matrix order unit suppose that  $\psi \in M_n(X^*)_{\rm sa}$ and $\psi \leq t e_n$ for all $t >0$.  Then $\psi \leq t \, I_n$ on $K$, hence $\psi \leq 0$ on $K$ and therefore also on $M_m(X_+)$ for all $m \in \bN$.
 So $\psi \leq 0$.  To see that $e$ is matrix archimedean suppose that  $\psi \in M_n(X^*)_{\rm sa}$.  We know that 
 $\psi$ is bounded uniformly by a constant $c = \| \psi \|_{\rm cb}$
 on all matrix unit balls from $X$.  Hence $\psi$ is bounded uniformly by $c$ on $K$.  This implies that $c e - \psi \geq 0$ on $K$ and hence also on $M_m(X_+)$ for all $m \in \bN$.
 This proves the Claim.

 Thus $(X^*,e)$ is a (dual)  operator system.  To see that the dual operator space matrix norms for $X^*$ agree with the matrix order unit norms we first suppose that $\psi \in M_n(X^*)_{\rm sa}$.  If $\psi$ is completely contractive then the argument in the last paragraph shows that 
 $\pm \psi \leq e$.  Conversely, suppose that $-e_n \leq \psi \leq e_n$.  Then $-I_n \leq \psi \leq I_n$ on $K$. 
 With notation from   3)\ of  Definition \ref{ncbns} suppose that  $x = c_1 x_1 c_1 - c_2 x_2 c_2$ with $c_i$ positive contractive matrices with $\| c_1^2 + c_2^2 \| < 1$ and $x_i \in K_m$.   Then $$\| \psi_m(x) \| = \| c_1 \psi_m(x_1) c_1 - c_2 \psi_m(x_2) c_2 \| = \| c_1^2 - c_2^2 \| \leq \| c_1^2 + c_2^2 \| <  1.$$  We may write a general $x \in {\rm Ball}(M_m(X))$ as a corner of a selfadjoint   $\tilde{x}$  in   ${\rm Ball}(M_{2m}(X))$ (with $2$-$1$ corner $x^*$ and other corners 0). Then  
 $\| \psi_m(x) \| = \| \psi_m(\tilde{x}) \| < 1$, which  shows that $\psi$ is completely contractive.
 
 Finally, for a general linear 
 $\psi : X \to M_n$, we may view $\psi$ as a corner of a selfadjoint  completely bounded 
map $\Psi$ on $E$ (with $2$-$1$ corner $\psi^*$ and other corners 0).  These have the same cb norm, and so by the last paragraph is  equal 
the order unit norm of $\Psi$.
 However this is the matrix order unit norm of $\psi$. 

 The real nc case  works just the same.
 \end{proof} 

  The second and third paragraphs of the last proof  work verbatim to show:

\begin{cor} \label{banbnn} If $F$ is a matrix base ordered space with its canonical matrix base norms, then  $F^*$  is a dual  operator system.  Moreover, this can be done with the dual operator space matrix norms for $F^*$ agreeing with 
the matrix order unit norms.
\end{cor} 

As in the classical case in Theorem \ref{xpc}, and by the same proof strategy, we can `complete' any matrix base ordered space to become a nc base norm space. 
If $E$ is a matrix base ordered space  with canonical nc base norms $\| \cdot \|_n$ and base $K$, then the closures $(\overline{K_n})$ are a nc base for $E$ with 
proper cones the closures of the matrix cones of $E$, and induce the same canonical nc base 
matrix norms $\| \cdot \|_n$.   This will be extremely important in 
constructing new examples of  nc base norm spaces in the future.

 \begin{cor}   \label{xpcnc} Let $X$ be a complex  matrix base ordered space  with canonical nc base norms $\| \cdot \|_n$ and base $K$. Then the closures $(\overline{K_n})$ are a nc base for $\bar{X}$, the completion of $X$, with 
proper matrix cones $\cD_n = \overline{(X_n)_+}$, the closures of the matrix cones of $X$, and induces the same canonical nc base 
matrix norms $\| \cdot \|_n$. Moreover $(\bar{X},(\cD_n))$ is a nc base norm space
which has the norm closure $\bar{K}$ of $K$ as a base, and which still has $(\| \cdot \|_n)$ as its 
 matrix base norms.  
\end{cor} 

\begin{proof} By Corollary \ref{banbnn} we have that $\cS = X^*$  is a dual  operator system with order unit $f_1$, the positive functional  corresponding to the nc hyperplane containing $K$.  Also, the dual operator space matrix norms for $X^*$ agree with 
the matrix order unit norms. 
Hence  $\bar{X}$ is a nc base norm space with its canonical predual matrix cones ${\mathfrak c} = ({\mathfrak c}_n)$ by Theorem \ref{banbn2}.   
We claim that ${\mathfrak c}_n = \overline{M_n(X)_+}$.  If this were false then by the Effros-Winkler geometric Hahn-Banach theorem \cite{EW}  there exists a continuous linear 
$f : X \to M_n$ with Re $f$ which is $\leq 0$ on all matrix cones
$M_m(X)_+$,  but not on 
${\mathfrak c}_n$.  Note that $f$ is completely bounded since its range is finite dimensional.
Let $g = -(f + f^*)$, then $g$ is completely positive with respect to $(\overline{M_m(X)_+})$, but is not positive 
on ${\mathfrak c}_n$.  
So $g \in M_n(\cS)_+$ so that $g$ is positive on ${\mathfrak c}_n$. 
This is a contradiction. 

Next we show that $\bar{K}$ is a nc base for ${\mathfrak c}$,
indeed that $\bar{K}  = (\{ x \in  {\mathfrak c}_n : (f_1)^{(n)}( x) = I \})$.  Suppose that $x \in {\mathfrak c}_n$ with $(f_1)^{(n)}( x) = I$.  Let  $x_k \in M_n(X)_+$ with
$x_k \to x$.  Then $(f_1)^{(n)}(x_k) \to I_n$, so that we may assume that $d_k = (f_1)^{(n)}(x_k)$ is invertible, and indeed that $d_k \geq \delta I$ for some fixed $\delta > 0$.
Let $y_k = d_k^{-\frac{1}{2}} x_k d_k^{-\frac{1}{2}}$, then 
 $(f_1)^{(n)}(y_k) = I$, so that $y_n \in K$.
 Since $\| d_k^{-\frac{1}{2}} - I \| \to 0$ by functional calculus,
 (or diagonalization of a positive matrix), we have $y_k \to x$.  So $x \in \bar{K}$.  That is, 
$\bar{K} = (\{ x \in  {\mathfrak c}_n : (f_1)^{(n)}( x) = I \})$.  
\end{proof} 

{\bf Remark.}  The real case of Corollary   \ref{xpcnc}  probably follows by the same proof.  We used the (complex) Effros-Winkler geometric Hahn-Banach theorem, but we did not take the time to check if the real version of that result is valid too, probably with the same proof. Or perhaps one could  just use a simple real geometric Hahn-Banach theorem argument instead.

\bigskip

 We obtain a characterization of nc base norm spaces:

 \begin{cor}   \label{ncbasis} A complete matrix ordered matrix normed  complex  (resp.\ real) $*$-vector space $X$ is a 
complex (resp.\ real)  nc base norm space with nc base $K$ if  and only if there is a dual operator system $V$ with operator  space predual $V_*$ and a $*$-linear completely isometric isomorphism $\theta : X \to V_*$ taking $K$ onto the normal matrix state space of $V$.
\end{cor} 

\begin{proof} For the nontrivial direction let $V = (X^*,e)$  be the 
dual operator system in the last proof.  Take $V_* = X$.  Moreover 
$$K_n = \{ x \in M_n(X)_+ : e_n(x) = I_n \} = {\rm UCP}^\sigma(V, M_n.)$$  \end{proof} 

Thus the generic nc base norm space ``is" the predual of a dual operator system (or equivalently of a dual $\bA(K)$-space).  For nc base spaces this is the analogue of Kadison's theorem. 
Similarly we shall see presently that the generic nc dual base norm space ``is" the dual of an  operator system.

\begin{prop} \label{rcnb} The (unique Ruan) reasonable complexification of a real nc base norm space $E$ (resp.\ real matrix dual base norm space) with nc base $K$ is a complex nc base norm space $E_c$ (resp.\ matrix dual base norm space) with nc base $K_c$. (See {\rm \cite{BMcI, BMcII}} for definitions of $E_c$, $K_c$, etc.)
\end{prop}

 \begin{proof}  Suppose that $E$ is a real nc base norm space, with 
 base $K$ and base function $f_1$.
 Then $E$ is an operator space and matrix ordered $*$-vector space, and $(f_1)_c$ is selfadjoint, where $(f_1)_c$ is the complexification of $f_1$. 
 Thus $E_c$ is an operator space and matrix ordered $*$-vector space, the latter by the Remark after \cite[Proposition 2.6]{BR}.  
 It is also `matrix normed':
 indeed $$\| (x + iy)^* \|_n = 
 \| c(x^*,-y^*) \|_n = \| c(x^*,-y^*)^* \|_n= 
 \| c(x,y) \|_n= \| x +iy \|_n,$$ for $x, y \in M_n(X)$.
 Since the matrix cones for $E$ are closed it is easy to argue that so are the matrix cones for $E_c$ from the above-mentioned Remark. 
 Moreover $(E_c)^* \cong (E^*)_c$ completely isometrically 
 and $*$-linearly.  However $(E^*,f_1)$ is a real operator system, so that  $((E^*)_c,f_1)$ is a complex operator system. 
 Hence $((E_c)^*, (f_1)_c)$ is a complex operator system.
 Thus $E_c$ is a complex base norm space with base $\{ x \in M_n(E_c)_+ : ((f_1)_c)^{(n)}(x) = I \}$ by 
 Lemma \ref{bn}.   If $x = y + iz \in M_n(E_c)_+$  then 
 $((f_1)_c)^{(n)}(x) = I$ iff $(f_1)^{(2n)}(c(y,z)) = I_{2n}$.
 By Lemma \ref{bn} again this is equivalent to $c(y,z) \in K_{2n}$,
 that is, if and only if    $x \in K_c$. 

 The dual case is similar.   For example that $K_c$ is compact follows because it is weak* closed and contained in the matrix ball. 
 \end{proof}

 {\bf Remark.} In general a complex operator system $V$ is not Ruan's `(unique) reasonable complexification' of $V_{\rm sa}$.  See \cite[Section 2.3]{BMcI}.  Similarly 
 a complex  nc base norm space $E$ with nc base $K$ is not Ruan's (unique) reasonable complexification of $E_{\rm sa}$.  Indeed
 the latter is a symmetric operator space (that is, $\| [x_{ij} ] \|_m = \| [x_{ji} ] \|_m$ always) by the discussion above \cite[Lemma 2.7]{BMcI}.

 \medskip

The following shows that the dual of a unital complex $C^*$-algebra is a complex nc dual  base norm space.  Indeed the operator space dual of an operator system is the generic nc dual  base norm space.   Equivalently, the nc dual  base norm spaces are exactly (up to appropriate isomorphism) the $\bA(K)^*$ for a compact nc convex set $K$.  The dual nc base of $\bA(K)^*$ is $\delta(K)$, where $\delta : K \to \bA(K)^*$ is the canonical map \cite{DK}.

\begin{thm} \label{banaou2} A matrix normed matrix ordered $*$-vector space $F$ is a 
complex  (resp.\ real) operator system, or equivalently an archimedean  matrix order unit space (with matrix norms agreeing with 
the matrix order unit norms) if  and only if $F^*$ is a  complex (resp.\ real) nc dual  base norm space,
and if  and only if $F^*$ is a  complex (resp.\ real) matrix dual  base norm space. 
\end{thm} 

\begin{proof} 
If $V$ is an operator system then $V^*$ is canonically 
a 
matrix ordered matrix normed  operator space with its canonical (closed) matrix predual cones.
The noncommutative state space $K$ is a nc dual base 
(hence a matrix dual base) for $V^*$.
Indeed by \cite[Lemma 2.2]{CE} every $u \in M_n(V^*)_+$ is of form $\alpha^* k \alpha$ for $k \in K_n, \alpha \in M_n$, indeed 
with $\alpha$ positive.  Combining this with the Wittstock decomposition, Ball$(M_n(E))_{\rm sa}$ consists of the expressions $x = c_1 x_1 c_1 - c_2 x_2 c_2$ with $c_i$ positive contractive matrices with $\| c_1^2 + c_2^2 \| \leq 1$ and $x_i \in K_n$.  Clearly $K_n = {\rm UCP}(V,M_n) = M_n(V^*)_+ \cap H_n$ where $H$ is the obvious hyperplane defined by $1 \in V$.  

Conversely, let $K$ be a matrix dual base (resp.\ nc dual base) for $E$.  By Definition \ref{dncbns}, $K$ is a compact matrix convex set (resp.\ nc convex set) in 
the selfadjoint part of the matrix unit balls of the dual $E = F^*$ of an operator space and $*$-vector space $F$.  
By 4) in that definition we may view $f_1 \in F_+$.
The map $\theta : F \to A(K)$ (resp.\ $\theta : F \to \bA(K)$), defined by $\varphi \mapsto \varphi|_K$ for $\varphi \in F$, is a completely contractive complete order embedding since if $\psi \in M_n(F)_{\rm sa}$ 
then $\psi \geq 0$ if  and only if $\psi \geq 0$ on $K$.  This is (by definition of nc base) because a selfadjoint $x \in M_m(E)$  is positive 
if and only if $x = c k c$ with $c$ a positive matrix and $k \in K_m$.   
Moreover $\theta(f_1) = 1$, and $\theta$ is completely isometric  since it is completely isometric on $F_{\rm sa}$.  
That it is completely  isometric on $F_{\rm sa}$ follows from definition of dual base,  that Ball$(M_n(E))_{\rm sa}$ consists of the expressions $x = c_1 x_1 c_1 - c_2 x_2 c_2$ with $c_i$ positive contractive matrices with $\| c_1^2 + c_2^2 \| \leq 1$  and $x_i \in K_n$.   For given $\psi \in$ Ball$(M_n(F))_{\rm sa}$ and given $\epsilon > 0$ we may choose such $c_i, x_i$ with 
$$\| \psi  \| - \epsilon <  \| \psi_n(c_1 x_1 c_1 - c_2 x_2 c_2) \| = 
\| c_1 \psi_n(x_1) c_1 - c_2 \psi_n(x_2)  c_2 \| \leq \| \theta_n(\psi) \|.$$  
For general $\psi \in {\rm Ball}(M_n(F))$ we may view $\psi$ as a corner of a selfadjoint  completely bounded 
map $\Psi$ on $E$ (with $2$-$1$ corner $\psi^*$ and other corners 0).
We have  $$\| \theta_{n}(\psi) \| =  \| \theta_{2n}(\Psi) \| = \| \Psi \| = \| \psi \|.$$  
  So $\theta$ is completely isometric. 
Thus because of the above selfadjoint completely  isometric  isomorphism into $\bA(K)$, $F$ is an operator system, indeed an archimedean matrix order unit $*$-vector space, with order unit a functional $f_1$ that is $I_n$ on $K_n$. 
\end{proof}

\begin{cor} If $K$ is a  nc dual base for 
a nc dual  base norm space $E$ then $K$ is nc regular in $E$ in the sense of {\rm \cite{Breg}}. 
\end{cor}

\begin{proof}   By the expression for Ball$(M_n(E))_{\rm sa}$ in 3) of Definition \ref{dncbns}, $E$ is the nc span of $K$, as we said  in the Remark above Lemma \ref{bn}. 
By 4) in Definition \ref{dncbns},   $K$ is contained in a hyperplane not passing through 0, and also $K \subset E_{\rm sa}$. So  $K$ is nc preregular in the sense of {\rm \cite{Breg}}.  
If $\bF = \bC$ then  $\theta : F \to \bA(K)$ in the last proof 
has dense range by   
\cite[Proposition 3.6]{Breg}, since $K$ is nc preregular.
In the real case one may use the Remark below
\cite[Theorem 4.3]{Breg} (using also Proposition \ref{rcnb}). 
So $\theta : F \to \bA(K)$ is a complete order isomorphism.  This means that 
$K$ is nc regular in $E$  \cite{Breg}. 
\end{proof} 

Conversely, for any nc regular embedding of $K$ in $E$ in the sense of \cite{Breg}, by the main results in that paper $E$ may be made into a dual nc base norm space, and its predual $F$ into an operator system. 

\begin{cor} \label{kd} If $K$ is a  matrix dual base for 
a matrix dual  base norm space $E$ then $K$ can be canonically augmented (by adding the strictly infinite levels to $K$) to be a nc dual base for  $E$, so that $E$ is a nc dual  base norm space. 
\end{cor}

\begin{proof}   By the proof of Theorem \ref{banaou2} the predual $F$ of $E$ is an operator system with matrix order unit $f_1$, and $K_n = \{ x \in M_n(E)_+ = CB^\sigma(F,M_n): (f_1)^{(n)}(x) = I_n \}$, for $n < \infty$.  This formula
also works to define  infinite levels of $K$, indeed this is just the normal nc state space.  
Alternatively, it is  the closed nc convex  hull of $K$ in $E$.  Indeed any nc compact convex set in a dual operator space is the closed nc hull of its finite levels.
This may  be seen for example from the following modification of an argument communicated to us by Matt Kennedy: Let 
$L$ be the closed nc hull of its finite levels.  Clearly $L \subset K$, and so $K_n = L_n$ for finite $n$. So $K = L$ by \cite[Proposition 2.2.10]{DK}.
\end{proof}

 \section{Some examples} \label{ex}
 
 \subsection{The Paulsen system $\cS_V$} \label{ex1}
 If $V$ is a (concrete) operator space in $B(H)$, the bounded operators on a Hilbert space $H$, then the {\em Paulsen system} of $V$ is the operator system given by
 $$\cS_V = \left\{ \left[\begin{array}{cc}\lambda I & x \\ y^* & \mu I \end{array} \right] : x,y \in V; \lambda, \mu \in \mathbb{C} \right\},$$
 where this is viewed as sitting inside $M_2(B(H)) \cong B(H \oplus H)$.  It is well-known that $\left[\begin{array}{cc} I & x \\ y^* & I \end{array} \right] \ge 0$ if and only if $x = y \in {\rm Ball}(M_n(X)$.   See \cite{Pnbook}.
 This formula also defines the {\em matrix order unit matrix norms}, namely $\| x \|_n \leq 1$ if and only if the last matrix is positive. 
Similarly, it is easy to see that  
\begin{equation} \label{plemma}\left[\begin{array}{cc}\lambda  & x \\ x^* & \mu \end{array} \right] \ge 0 \iff |\langle x \zeta, \eta \rangle| \le \sqrt{\langle \lambda \zeta, \zeta\rangle \, \langle \mu \eta ,\eta\rangle} \le \frac{1}{2}(\langle \lambda \zeta, \zeta \rangle + \langle  \mu \eta ,\eta \rangle) 
\end{equation}
for $\eta, \zeta \in H, \lambda, \mu \in M_n$ (see Eq.\ (1.24) in \cite{P}).
 Thus if $x \in M_n(X)$ and $\lambda, \mu \in M_n$ such that $\left[\begin{array}{cc}\lambda I & x \\ y^* & \mu I \end{array} \right] \ge 0$, then $\|x\| \le \sqrt{\|\lambda\|\|\mu\|}$.

 In \cite{Ng} W.H. Ng shows that the matrix ordered dual $(\cS_V)^d$ of an arbitrary Paulsen system $\cS_V$ is again an operator system. Indeed his calculations show that the normalized trace $\tau : \cS_V \rightarrow \bC$, given by $\tau\left(\left[\begin{array}{cc}\lambda I & x \\ y^* & \mu I \end{array} \right]\right) = \frac{1}{2}(\lambda+\mu)$, is an archimedean matrix order unit for $(\cS_V)^*$.
Let us write $(\cS_V)^d$ for Ng's dual system, which is an operator space with its matrix order unit norm.  In the light of Theorem \ref{banbn2} one might thus wonder if $\cS_V$ is an nc base norm space.  The nc base here is the collection of sets $K_n = (\tau^{(n)})^{-1}(I_n) \cap M_n(\cS_V)_+$.  Indeed Theorem \ref{banaou2} shows immediately that $\cS_V$, with the operator space dual structure coming from $((\cS_V)^d)^*$ is a nc base norm space when $V$ is finite dimensional.  For then
$\cS_V$ is the operator space (and matrix ordered) predual of
$(\cS_V)^d$ (c.f.\ Proposition \ref{ncbnsf}). 

The following is no doubt well known but we cannot locate a proof in the literature.

\begin{lem} \label{ps} For a general operator space $V \subset B(K,H)$ any positive matrix  in $M_n(\cS_V)$ may be written
as 
$$\left[ \begin{array}{cc} \lambda \otimes I_H & x \\ x^* & \mu \otimes I_K \end{array} \right] 
= (\lambda^{\frac{1}{2}} \oplus \mu^{\frac{1}{2}}) \, \left[ \begin{array}{cc} I & z \\ z^* & I \end{array} \right]
\, (\lambda^{\frac{1}{2}} \oplus \mu^{\frac{1}{2}}),$$
with $z \in {\rm Ball}(M_n(V))$, and  $\lambda, \mu \in M_n^+$.
Indeed we can choose \newline $z = (e \lambda e)^{-\frac{1}{2}} \, x \, (f \mu f)^{-\frac{1}{2}}$ where $e = s(\lambda), f = s(\mu)$. 
\end{lem}

\begin{proof}  If $\lambda$ and  $\mu$ are invertible then this is well known, with $z = \lambda^{-\frac{1}{2}} x \mu^{-\frac{1}{2}}$ \cite{Pnbook}.  Indeed this follows easily from (\ref{plemma}) above. Otherwise, if $e = s(\lambda)$ is the support projection
of $\lambda$  then $e \lambda e$ is invertible 
in $B(e \ell^2_n)$;  Similarly for  $\mu$ and $f$.
Identifying $e$ with $e \otimes I_H$ as usual, (\ref{plemma}) 
implies that $x f^\perp = e^\perp x = 0$, so that 
$x = exf$ and $a = (e \oplus f) a (e \oplus f)$.  
Cutting down to these supports, i.e.\ replacing $\ell^2_{n} \oplus \ell_n^2$ by 
 $(e \oplus f) \ell^2_{2n}$, we may assume that $\lambda$ and  $\mu$ are invertible.    The statement is now evident, with $z = ez f = (e \lambda e)^{-\frac{1}{2}} \, x \, (f \mu f)^{-\frac{1}{2}}$. 
\end{proof} 

\begin{thm} \label{ncbnst} For a general operator space $V$ we have 
that $\cS_V$ has an operator space structure with respect to which it is a  nc base norm space with nc base $K$ and $f_1 = \tau$.  Moreover the nc base matrix norms are equivalent to the 
original matrix norms of $\cS_V$. Finally, the operator space dual of this nc base norm space is Ng's operator system $(\cS_V)^d$. \end{thm}

\begin{proof} We first claim that $\cS_V$ is a matrix base ordered space in the sense defined above Theorem \ref{doesd}, with nc base $K$ and $f_1 = \tau$.  Note that $f_1$ is 
selfadjoint and strictly positive since if $\lambda, \mu \geq 0$ and $\lambda + \mu = 0$ then $\lambda =  \mu = 0$, so that 
$x = 0$ too.  Moreover if $\tilde{x} \geq 0$ then 
$x = 0$ (a well known matrix fact).   

We  observe that the sets $K_n$ have a sharp upper bound of $4$, with respect to the original operator system norm.  To see that $4$ is a bound, note that any element of $K_n$ can be written in the form $\left[\begin{array}{cc} \lambda & x \\ x^* & \mu \end{array} \right]$, where $x \in M_n(V)$ and $\lambda$ and $\mu$ are positive matrices such that $\lambda + \mu =  2I$.  It follows that both $\lambda$ and $\mu$ have norm less than or equal to $2$. Then $$\left\|\left[\begin{array}{cc} \lambda & x \\ x^* & \mu \end{array} \right]\right\|\le \left\|\left[\begin{array}{cc} \lambda & 0 \\ 0 & \mu \end{array} \right]\right\| + \left\|\left[\begin{array}{cc} 0 & x \\ x^* & 0 \end{array} \right]\right\| \le 2 + \|x\| \le 2 + \sqrt{\|\lambda\| \|\mu\|} \le 4.$$ Finally, it is easy to find elements at all matrix levels that have a norm equal to $4$.

By definition $K_n = \{ x \in M_n(X)_+ : (f_1)^{(n)}(x) = I_n \}$.    That  $E$ is based on $K$  is immediate from Lemma \ref{ps}.
Since $\cS_V$ is an operator system any 
$w \in M_n(\cS_V)_{\rm sa}$  may be written as $(\| w \| I + w)- (\| w \| I + w)$.  Thus 
$w = \alpha_1^*x_1\alpha_1 - \alpha_2^* x_2 \alpha_2$ as desired.  Suppose that $w$ may be written in the latter form but with $\max \{ \| \alpha_1 \| , \| \alpha_2 \| \} < \epsilon$. 
Then by the calculation e.g.\ in the paragraph after Theorem \ref{doesd}, we have that the original norm of $w$ is dominated by  $$\| \alpha_1^* \alpha_1 + \alpha_2^* \alpha_2
\| \, \max \{ \| x_1 \|_{\rm os} , \| x_2 \|_{\rm os} \}
\le 2 \epsilon^2 \cdot 4 .$$  Here $\| x_i \|_{\rm os}$ is
the original operator system norm of $x_k \in K_n$.  So $\| w \|_{\rm os} = 0$ and $w = 0$.  Thus we have verified that 
$\cS_V$ is a matrix base ordered space in the sense defined above Theorem \ref{doesd}, with nc base $K$ and $f_1 = \tau$.
By that theorem, $\cS_V$ is an operator space and matrix ordered matrix normed  $*$-vector space.  Let us write ${\mathfrak S}(V)$ for $\cS_V$ with this new operator space structure, with respect to which it is a matrix base ordered norm space.   Of course as a matrix ordered space ${\mathfrak S}(V) = \cS_V$.  Theorem \ref{doesd} also informs us that $f_1$ is contractive on ${\mathfrak S}(V)$, and 
${\mathfrak S}(V)$ satisfies all of the conditions to be a nc base norm space
except possibly for $M_n(X)_+$ and $K_n$ being closed.  However 
we shall show shortly that the nc base norm is equivalent to the 
original matrix norms of $\cS_V$, so that $M_n(X)_+$ and $K_n$ are closed.    Thus
  ${\mathfrak S}(V)$ is a nc base norm space.

We next show that the nc base norm $\|\cdot\|_{\tau}$ is equivalent to the original operator system  norm at all matrix levels, by showing this is the case first on the positive elements, the selfadjoint elements, and then for a general element.
Suppose $u = \left[\begin{array}{cc} \lambda & x \\ x^* & \mu \end{array} \right]$ such that $\|x\| \le \sqrt{|\lambda  | |\mu|}$, and $\lambda, \mu \ge 0$.  Assume $|\lambda| \geq |\mu|$. Then $$\tau (u) \le |\lambda | \le \|u\| \le \|\lambda I \oplus \mu I \| + \left\| \left[\begin{array}{cc} 0 & x \\ x^* & 0 \end{array} \right] \right\| = |\lambda | + \|x\| \le |\lambda | + \sqrt{|\lambda  | |\mu|} \le   2 |\lambda|,$$
and this is dominated by $4 \, \tau(u).$ 
  A similar calculation, but with $| \cdot |$ replaced by the matrix norm, shows that on the positive cone of $M_n(S_V)$, the new matrix norm is equivalent to the original matrix norm.  

  Next, 
fix selfadjoint nonzero $x \in M_n(S_V)$, and write 
$\| \cdot \|_{\tau,n}$
for 
the base norm at the $n$th matrix level.  So for any $t>1$ there are $x_i \in K_n$ and positive matrices $c_1$ and $c_2$ such that $x
= c_1 x_1 c_1 - c_2 x_2 c_2$ and $\| c_1^2 + c_2^2 \| \leq t \|x\|_{\tau,n}$.  Then
\begin{eqnarray*}\|x\|_n & = &
\|c_1 x_1 c_1 - c_2 x_2 c_2\|_n \\
    & = &
    \| [c_1 \; c_2 ] (x_1 \oplus (-x_2)) [c_1 \; c_2 ]^\tran \|_n \\
    & \le & 
    \| c_1^2 + c_2^2 \| \|x_1 \oplus (-x_2)\|_{2n} \\
    & \le & t \, \|x\|_{\tau,n} \, \max\{\|x_1\|_n,\|x_2\|_n\} \le 4t \, \|x\|_{\tau,n},
\end{eqnarray*}
since $4$ is the upper bound on the original norm of  $K_n$, as noted earlier. Taking the infimum over $t>1$ gives the inequality $\|x\|_n \le 4\|x\|_{\tau,n}.$

For the other direction, we again consider a selfadjoint element $u$ in the unit ball of $M_n(S_V)$ (with respect to the original norm on $S_V$), which we may assume can be written $u = \left[\begin{array}{cc}\lambda & x \\ x^* & \mu \end{array} \right]$ with $x \in M_n(V)$ and $\lambda$ and $\mu$ scalar matrices. To show that the nc base norm of such an element has a bound, it suffices to show that each of $\left[\begin{array}{cc}\lambda & 0 \\ 0 & \mu\end{array} \right]$ and $\left[\begin{array}{cc}0 & x \\ x^* & 0 \end{array} \right]$ have a bound on their nc base norm.  For the first matrix, since it is selfadjoint, there are contractive positive matrices $p_i,q_i$ ($i=1,2$) such that $\lambda \oplus \mu  = (p_1-p_2) \oplus (q_1-q_2)$, and $\| \lambda \| = \| p_1 + p_2\|, \| \mu \| = \| q_1 + q_2\|.$  We have 
$$\lambda \oplus \mu  = (p_1^{1/2} \oplus q_1^{1/2}) (I \oplus I) (p_1^{1/2} \oplus q_1^{1/2}) - (p_2^{1/2} \oplus q_2^{1/2}) (I \oplus I) (p_2^{1/2} \oplus q_2^{1/2}).$$  From this, and the definition of the nc base norm we see that $\|\lambda \oplus \mu\|_{\tau, 2n} \le \|(p_1\oplus p_2) + (q_1\oplus q_2)\| \le 1$.
For the other matrix, we have $$\left[\begin{array}{cc}0 & x \\ x^* & 0 \end{array} \right] = \frac{1}{2}\left[\begin{array}{cc}I & x \\ x^* & I \end{array} \right] - \frac{1}{2}\left[\begin{array}{cc}I & -x \\ -x^* & I \end{array} \right],$$
which has nc base norm less than or equal to $1$. Thus, $\|u\|_{\tau, n}  \le 2 \|u\|_n$.

Finally for an arbitrary matrix, the corner trick used many times in the proof of Theorem  \ref{banaou2} shows that the nc base norms are equivalent to  the original matrix norms.  Indeed, if $u \in M_n(S_V)$, then $\tilde{u} = \left[\begin{array}{cc}0 & u \\ u^* & 0 \end{array} \right]$ is selfadjoint, so that $\frac{1}{2}\|\tilde{u}\|_{\tau,2n} \le \|\tilde{u}\|_{2n} \le 4\|\tilde{u}\|_{\tau,2n}$. Thus $\frac{1}{2}\| u \|_{\tau,n} \le \|u \|_n \le 4\| u \|_{\tau,n}$. Thus the nc base norm is
equivalent to the original matrix norms of $\cS_V$. 

To see that the operator space dual ${\mathfrak S}(V)^*$  is Ng's operator system $(\cS_V)^d$, we simply appeal to the fact in the proof of Theorem \ref{banbn2} that the operator space dual 
matrix norms on $X^*$ agree with the matrix order unit norm induced by the matrix order unit $f_1$. 
  \end{proof}

{\bf Remarks.} 1)\ Explicitly, we have $$K_1 = \left\{ \left[\begin{array}{cc}\lambda I & x \\ x^* & (2-\lambda) I \end{array} \right] : x \in V; 0 \le \lambda \le 2, \|x\| \le \sqrt{\lambda(2-\lambda)} \right\},$$
  where $\| \cdot \|$ denotes the original norm on $S_V$.

\smallskip

  2)\ Note that the nc base norm on ${\mathfrak S}(V)$ is not the original norm. For example, the nc base norm on a diagonal matrix $\lambda I \oplus \mu I$ is $\frac{1}{2}(\lambda + \mu) \neq \max(\lambda,\mu)$, if $\lambda, \mu \geq 0$ are distinct, by Lemma \ref{isa} with $f_1 = \tau$.  If $x \in M_n(V)$ has (original) norm (less than or) equal to $1$, then $\left[\begin{array}{cc} I & x \\ x^* & I \end{array} \right]$ has nc base norm equal to $1$.   This also shows that the nc base norms
  restricted to the `off-diagonal' are dominated by the original matrix norms.

\bigskip 

  One may call ${\mathfrak S}(V)$ constructed here the {\em base Paulsen system} of an operator space $V$.

 \subsection{Other examples}\label{ex2}
 
\begin{prop} \label{ncbnsf} For a general finite dimensional operator system $\cS$ and faithful state $f_1$ on $\cS$ we have 
that $\cS$ has an operator space structure with respect to which it is a  nc base norm space with nc base $K = (K_n) = 
(\{ x \in M_n(X)_+ : f_1^{(n)}(x) = I_n \})$ and base function $f_1$.  Moreover the nc base matrix norms are equivalent to the 
original matrix norms of $\cS$. Finally, the operator space dual of this nc base norm space is the  operator system dual $\cS^d$. \end{prop}

\begin{proof}  Choi and Effros prove in \cite[Corollary 4.5]{CE} that if $f_1$ is a faithful state on a finite 
dimensional operator system $\cS$ then  $\cS^d = \cS^*$  is an operator system with identity/order unit $f_1$.  See \cite[Section 8]{BR} for the real case of this.   Theorem \ref{banaou2} shows immediately that $\cS$, with the operator space dual structure coming from $(\cS^d)^*$, is a  nc base norm space with nc base $K = (K_n) = (CB^\sigma(\cS,M_n))
= 
(\{ x \in M_n(\cS)_+ : f_1^{(n)}(x) = I_n \})$.  For then
$\cS$ is the operator space (and matrix ordered) predual of
$\cS^d$.  The matrix norm equivalence follows since any 
 isomorphism of finite dimensional operator spaces is
a complete isomorphism \cite{ER}. 
\end{proof}

{\bf Remark.}  A similar argument to Theorem \ref{ncbnst},  using Theorem \ref{doesd}, should work to show that several other known (infinite dimensional) operator systems $\cS$ whose dual $\cS^d$ has an archimedean  matrix order unit (so that  $\cS^d$ can be realized as an operator system),
are a nc base norm space which is also an operator space predual of the original operator system.

 \medskip

 The following gives a new approach to complex base norm spaces, using the Min and Max functors in operator space theory
 (see e.g.\ 1.2.21 and 1.2.22 in \cite{BLM}).  The dual Taylor norm condition disappears in this formulation, but it is replaced by matrix norm conditions.    

\begin{cor} A complete ordered complex  normed $*$-vector space $V$ is a complex base norm space if and only if ${\rm Max}(V)$ is a nc base norm space.   And $X^*$ is a complex dual base norm space if and only if ${\rm Max}(V^*)$ is a dual nc base norm space.
\end{cor}

 \begin{proof} Note that $V^*$ is a complex dual aou space iff
 Min$(V^*)$ is a dual operator system. So iff Max$(V)$ is a nc base norm space, by Theorem \ref{banbn2} and the duality of Min and Max. But $V^*$ being a complex dual aou space is also equivalent to $V$ being a base norm space by Theorem \ref{banbn}.

 The other is similar.  Indeed  $V$ is a complex normed aou space iff  Min$(V)$ is an operator system.  So iff Max$(V^*)$ is a dual nc base norm space.  But this is also equivalent to $V^*$ being a dual base norm space. 
 \end{proof}

{\bf Remark.}  1)\ The above is related to the perspective of \cite{PTT} for aou spaces (and other works building on that paper), and therefore will have implications for entanglement. 

\medskip 

2)\ The real case of the last result is probably easily checked to be true, by complexification, or by the same argument 
but using \cite[Proposition 2.6]{Sharma} and \cite[Theorem 3.5]{BR} (see also \cite[Section 9]{BR}).

\bigskip

{\bf Acknowledgement}: This project was partially supported by NSF grant DMS-2154903. We thank Fred Schultz and Vern Paulsen for some historical remarks
(on real base norms and Lemma \ref{ps} respectively), and Matt Kennedy for some comments and conversation
(e.g.\ on a fact in the proof of Corollary \ref{kd}).   The second author would also like to thank Ricky Ng for discussions on his work on the dual of a Paulsen system, which inspired the initial inquiry into some of the topics treated here. 
 After we circulated our paper Travis Russell pointed out his paper \cite{Rus}.
The published version of this contains a new Section 7
which contains an interesting   `noncommutative base normed space' variant, and their {\em ordered} duality with
operator systems.   However it is stated there that for an operator system $S$, 
``The matrix norm on $S^d$ induced by its gauge is generally different from the operator space dual'', that is $S^*$.  Indeed if $S = l^\infty_2$ (an example suggested by him) 
then his $S^d$ base space is not $l^1_2$, thus is not a strict generalization of the classical base norm theory.  
Indeed the nc base matrix norms of \cite[Section 7]{Rus} are different to (indeed are equivalent up to a constant, and are  dominated by) ours.   Also his noncommutative base norms on his noncommutative base spaces induce operator space dual norms which do not agree with the matrix order unit norms on the dual.  Nonetheless, studying his definitions helped us to shorten our nc base definition.

Suppose that $E$ is a `noncommutative base normed space' in his sense, and  $x \in M_n(E)_{\rm sa}$.   Suppose that for 
every  $\epsilon > 0$ we can write $x = y-z$ for $y, z \in M_n(E)_+$ with $\| f_1^{(n)}(y + z) \| < \epsilon$.  Then $x \leq y \leq \|  f_1^{(n)}(y) \| \, k < \epsilon \, k$ for some $k \in K_n$
by Lemma  \ref{ts} (1).  So $x \leq 0$.
Similarly,   $-x \leq \epsilon \, k'$, and $x \geq 0$.   Thus $x = 0$.  Thus  $E$ is a matrix base ordered space in our sense. 
That his norms $(\| \cdot \|^\nu_n)$ are dominated by ours follows from Lemma \ref{ts}.  By 3) in Definition \ref{ncbns}, any $\varphi \in M_n(X)_{\rm sa}$ with our norm $< 1$ is dominated by some $y\in M_n(X)_+$ with our norm $< 1$,
hence is dominated by some $k \in K_n$.  Conversely if $0 \leq \varphi \leq k \in K_n$ then
$\| \varphi \| = \| (f_1)^{(n)}(\varphi)
\| \leq 1 .$  Thus the two norms are actually equal on $M_n(X)_+$. That they are equivalent up to a constant follows from this.  Suppose that $x \in M_n(X)_{\rm sa}$ with $t > \| x \|^\nu_n$ and 
$x \leq t k$ and $-x \leq t k$.  Then 
$$\| x \| \leq \| x + t k \| + t \| k \|
= \| x + t k \|^\nu_n + t \leq 3 t.$$

\end{document}